\documentclass{article}

\usepackage[utf8]{inputenc}
\usepackage[english]{babel}
\usepackage[numbers,sort&compress]{natbib}
\usepackage{amsmath,amssymb,amsthm}
\usepackage{hyperref}
\usepackage{multicol}
\usepackage{multirow}
\usepackage{mathtools}
\usepackage{booktabs}
\usepackage{graphicx,caption,subcaption}
\usepackage{float}

\usepackage{authblk}
\usepackage{nicefrac}
\usepackage[vlined,ruled]{algorithm2e}
\usepackage{etoolbox}
\usepackage[left=3cm, right=3cm, top=3.5cm,bottom=3.5cm]{geometry}

\usepackage{csquotes}

\providecommand{\keywords}[1]
{
	\small
	\textbf{\textit{Keywords---}} #1
}
\hypersetup{colorlinks, citecolor=red}

\newtheorem{definition}{Definition}
\AfterEndEnvironment{definition}{\noindent\ignorespaces}
\newtheorem{theorem}{Theorem}
\theoremstyle{definition}
\newtheorem{example}{Example}

\newtheorem{remark}{Remark}
\newcommand*\samethanks[1][\value{footnote}]{\footnotemark[#1]}

\newcommand\fract[5]{\ensuremath{\prescript{#1}{#3}{\mathcal{#5}}_{#4}^{#2}}}

\begin{document}
	\flushbottom
	\title{An Orthogonal Polynomial Kernel-Based Machine Learning Model for Differential-Algebraic Equations}

\author[1]{Tayebeh Taheri \thanks{Email: \{ttaherii1401, alirezaafzalaghaei\}@gmail.com}}
 
\author[1]{Alireza Afzal Aghaei\samethanks}	
 \author[1,2,3]{Kourosh Parand \thanks{Email: k\_parand@sbu.ac.ir, Corresponding author}}

	\affil[1]{\small{Department of Computer and Data Sciences, Faculty of Mathematical Sciences, Shahid Beheshti University, G.C. Tehran, Iran}}
	\affil[2]{\small{Department of Cognitive Modeling, Institute for Cognitive and Brain Sciences, Shahid Beheshti University, G.C. Tehran, Iran}}
	\affil[3]{\small{Department of Statistics and Actuarial Science, University of Waterloo, Waterloo, Canada}}

	\maketitle

	\begin{abstract}

The recent introduction of the Least-Squares Support Vector Regression (LS-SVR) algorithm for solving differential and integral equations has sparked interest. In this study, we expand the application of this algorithm to address systems of differential-algebraic equations (DAEs). Our work presents a novel approach to solving general DAEs in an operator format by establishing connections between the LS-SVR machine learning model, weighted residual methods, and Legendre orthogonal polynomials. To assess the effectiveness of our proposed method, we conduct simulations involving various DAE scenarios, such as nonlinear systems, fractional-order derivatives, integro-differential, and partial DAEs. Finally, we carry out comparisons between our proposed method and currently established state-of-the-art approaches, demonstrating its reliability and effectiveness.
	\end{abstract}

	\keywords{Least-squares support vector regression, Differential-algebraic equations, Fractional derivative, Machine learning methods, Weighted residual methods}

 	\section{Introduction}
 	A system of differential-algebraic equations (DAEs) is a combination of differential equations and
algebraic equations, in which the differential equations are related to the dynamical evolution of the
system, and the algebraic equations are responsible for constraining the solutions that satisfy the
differential and algebraic equations. DAEs serve as essential models for a wide array of physical phenomena. They find applications across various domains such as mechanical systems, electrical circuit simulations, chemical process modeling, dynamic system control, biological simulations, and control systems. Consequently, solving these intricate differential equations has remained a significant challenge for researchers. To address this, a range of techniques including numerical, analytical, and semi-analytical methods have been employed to tackle the complexities inherent in solving DAEs.
Examples of these techniques include Adomian Decomposition Method (ADM) \cite{celik2006solution}, Pade series \cite{ccelik2003numerical}, Runge-Kutta \cite{butcher1998efficient}, variational iteration \cite{soltanian2009application}, homotopy perturbation \cite{soltanian2010solution}, power series method \cite{benhammouda2014analytical}, homotopy analysis method \cite{zurigat2010analytical}, spectral methods \cite{ahmed2020numerical}, and so on. Table \ref{tab_comparison} presents a comparison of properties among spectral, deep learning, and LS-SVR methods.

 In recent years, machine learning schemes have been developed to solve various differential equations, such as ordinary differential equations (ODEs), partial differential equations (PDEs), stochastic differential equations, and more. These innovative approaches try to extract the problem's underlying structure by learning from the given differential equations. The support vector machine technique is one of the most powerful machine learning methods that, according to a set of labeled training examples, tries to increase the generalization of the learned solution. Many scientists and researchers have used this method to solve differential equations. For instance, Hajimohammadi et al. \cite{hajimohammadi2020new} proposed a numerical learning method to solve the general Falkner–Skan model. A least-squares support vector regression scheme has been developed for solving differential and integral equations by Aghaei and Parand \cite{aghaei2023hyperparameter}. Lu et al. \cite{lu2020solving} solved a higher nonlinear ordinary differential equation via a method based on least-squares support vector machines. 
 In the mentioned articles, we see instances of the use of support vector regression to solve different types of differential equations. However, until now, this method \cite{taheri2023bridging} has not been used to solve a system of differential-algebraic equations. According to the noted statements, in this research, we seek to develop this methodology for solving systems of differential-algebraic equations. A significant point of consideration within this paper is that these equations have been solved with fractional derivative orders, the solutions of which are considerably more intricate in comparison to equations involving integer derivatives. Since a system of equations combines both differential and algebraic components, the process of resolving DAE systems inherently entails a higher level of complexity when compared to the task of solving ODE systems. In this paper,  the application of weighted residual methods (WRM) to approximate solutions of differential equations is utilized. A cohesive integration of the collocation approach and the least-squares support vector regression method has been employed to determine the weight coefficients. Notably, within this framework, the adoption of Legendre functions as kernel functions in the context of the least squares support vector regression methodology is a distinctive facet. The proposed methodology has been employed to achieve numerical solutions for various categories of systems of DAEs, including fractional DAEs, integro-differential algebraic equations of fractional order, and partial differential-algebraic equations.

 In brief, the distinguishing features of our work encompass the following points:
\begin{itemize}
    \item Merging the collocation method with the LS-SVR approach.
    \item Utilizing Legendre functions as kernels within the LS-SVR algorithm.
    \item Utilizing LS-SVR for addressing linear and nonlinear DAEs involving fractional derivatives.
    \item Utilizing LS-SVR for addressing Volterra integro-differential algebraic equations of fractional order.
    \item Utilizing LS-SVR for addressing partial DAEs.

\end{itemize}

\begin{table}[htbp]
        \centering
        
\begin{tabular}{|p{0.25\textwidth}|p{0.2\textwidth}|p{0.25\textwidth}|p{0.25\textwidth}|}
\hline 
 Property & Spectral Methods & Deep Learning & LS-SVR \\
\hline 
  Convergence & {\footnotesize Exponential Convergence.} &  {\footnotesize Ensuring convergence to the best solution is not always guaranteed.} & {\footnotesize Exponential Convergence.} \\
\hline 
 Time Complexity & {\footnotesize Solving a linear system of equations.} & {\footnotesize  Usually computationally expensive, especially for large models.} &  {\footnotesize  Solving a positive definite linear system of equations.} \\
\hline 
  Parameters & {\footnotesize  Few parameters.} & {\footnotesize Parameters, including weights and biases, can be numerous} & {\footnotesize  Few parameters. } \\
\hline 
  Scalability& {\footnotesize Calculations become challenging when dealing with differential equations in high dimensions.} & {\footnotesize Scalable for high dimensions and complex domains.}  & {\footnotesize Calculations are superior to spectral methods but fall short of the performance achieved by deep learning.}  \\
 \hline
\end{tabular}
        \caption{Comparison of Methods}
        \end{table}\label{tab_comparison}

  The rest of this paper is organized as follows. In section \ref{sec2}, first, some types of differential-algebraic equations are reviewed, and a brief explanation of fractional derivative calculations is given. Section \ref{sec3} introduces the least-squares support vector machine regression model, following a presentation of orthogonal polynomials and their characteristics. The efficacy of our approach is demonstrated through its implementation in approximating various examples, as detailed in section \ref{sec4}. Finally, a summary of the proposed method is given in section \ref{sec5}.

 	\section{Preliminaries} \label{sec2}
 	This section explains some algebraic-differential equations solved in the numerical examples section. Ultimately, we introduce Caputo’s fractional derivative for differential equations with fractional derivative.
 	\subsection{A system of differential-algebraic equations}

 A system of differential-algebraic equations is a mathematical model where certain variables are described by differential equations, while others are constrained by algebraic equations.

 The general form of this case is:

\begin{align}
 &f(t, X(t),X^{\prime}(t),Y(t))=0,
 \\
&g(t,X(t),Y(t))=0,
\end{align}
    where $t$ is independent variable, $X\in \mathbb{R}^n$ is a vector of differential variables, $X^{\prime}$ is derivative of $X$, $Y\in \mathbb{R}^m$ is the vector of algebraic variables, function $f$ is a set of algebraic equations, and function $g$ describes the algebraic constraints of the system. These equations are known based on the index of the DAE. The index is the highest derivative that the differential variables should be differentiated. The form of the DAE is also classified based on the type of algebraic equations it includes, such as semi-explicit and fully implicit DAEs. In Section \ref{sec4}, we address the resolution of an index-1 semi-explicit DAE \cite{hosseini2006adomian} within Example \ref{ex1}. The general form of these kinds of differential equations is given by:
    \begin{equation}
\begin{cases}
X^{\prime}(t) =f(t,X(t),Y(t))\\
g(t,X(t),Y(t))=0. \\

\end{cases}.
\end{equation}

\subsection{Integro-differential algebraic equations}

Integro-differential algebraic equations (IDAEs) combine differential equations and algebraic equations with integrals. In this system, some equations have integrals of the dependent variables. The general form of this case can be expressed as \cite{zolfaghari2021structural}:
\begin{equation}\label{idae}
    f(t,X^{(v)}(t),Y(t))+\int_{a}^{t}h(t,s,X^{(\leq v)},Y(s))ds=g(t,X(t),Y(t))=0,
\end{equation}
where $t\in [a,t_f]$ is the independent variable and with dependent variable $X$ which their derivatives up to $v$-th order. The equation (\ref{idae}) is made of one section of differential-algebraic
equation, $f(t,X^{(v)}(t),Y(t))$, and one section of integral
algebraic equation, $h(t,s,X^{( v)},Y(s)$.
\begin{itemize}
    \item If $h=0$, the equation (\ref{idae}) will be a system of differential–algebraic equations.
    \item If $f=0$,  the equation (\ref{idae}) will be a integral
algebraic equation.
\item If $v=0$, the equation (\ref{idae}) will be a system of integral–algebraic equations.
\end{itemize}

In section \ref{sec4}, within example \ref{ex2}, a case study focuses on solving a Volterra integro-differential algebraic equation (VI-DAE) with weakly singular kernels of index-1.

\subsection{Partial differential-algebraic equations}

Partial differential-algebraic equations (PDAEs) are a class of mathematical equations that involve one partial differential equation, one ordinary differential equation, and one algebraic equation. Partial differential equations have partial derivatives related to the spatial coordinates. However, PDAEs have derivatives concerning both spatial and time coordinates. This makes PDAEs much more complex than PDEs and requires a different approach. This is a form of linear partial differential-algebraic equations \cite{gunerhan2020analytical}:

\begin{equation}\label{pdae}
    Av_t(t,x)+Bv_{xx}(t,x)+Cv(t,x)=g(t,x),
\end{equation}
where $A,B,C\in \mathbb{R}^{n\times xn}$ are constant matrices that can be singular, $t\in(0,t_e),x\in(-l,l)\subset\mathbb{R}$, and $v,g:[0,t_e]\times[-l,l]\rightarrow\mathbb{R}^n$.
\begin{itemize}
    \item If $A=0$, the equation (\ref{pdae}) will be an ordinary differential equation.
    \item If $B=0$, the equation (\ref{pdae}) will be a differential-Algebraic equation.
\end{itemize}

In section \ref{sec4}, as illustrated in example \ref{ex5}, a partial differential-algebraic equation has been approximated. In the next part DAEs with fractional order will be explained.

\subsection{Fractional differential-algebraic equations}

Fractional differential-algebraic equations (FDAEs) involve fractional derivatives, which can lead to precise and adaptable for specific physical systems to anticipate and regulate their behavior. Section \ref{sec4} covers diverse categories of FDAEs, encompassing linear, nonlinear, and Volterra integro-differential algebraic equations. In general, an FDAE can be written in the following form:
\begin{align}
 &\fract{}{\alpha}{}{}{D} x(t) = f(t, X(t),D^v X(t), Y(t))
\\
 &g(t, X(t), D^{v} X(t), Y(t)) = 0,
\end{align}
where $\fract{}{\alpha}{}{}{D}$ denotes the fractional derivative of order $\alpha$ which may be given in terms of Caputo or Riemann-Liouville fractional derivative. In the next part, a brief explanation of calculating the fractional derivative is provided.

 	\begin{definition}
 	  Consider a continuous function $h$ over the interval $[a,b]$. The Riemann-Liouville fractional derivative of any order $\alpha$ is defined as follows:
 	  \begin{equation}
 	       \fract{}{\alpha}{}{}{D}h(x)=\frac{1}{\Gamma(m-\beta)}\frac{d^m}{dx^m}\int_{a}^{x}(x-t)^{m-\alpha-1}h(t)dt, \quad m-1\leq \alpha \leq m,
 	  \end{equation}
 	here $\Gamma(.)$  represents the Gamma function, symbolized by the subsequent integral:
 	    \begin{equation}
 	  \Gamma(z)=\int_{0}^{\infty} e^{-\tau} \tau^{z-1}d\tau,
    \end{equation}
 that this integral is convergent in the right half of the complex plane $(\text{Re}(z)\geq 0)$.
 	  \end{definition}

          Caputo's \cite{caputo1967linear} formulation of non-integer order differentiation presents a distinctive approach to the concept. A notable advantage of this definition is its ability to express initial conditions of fractional differential equations in a format consistent with those for integer-order differential equations. The following definition offers an alternative approach to formulating the fractional derivative.

 	   \begin{definition}
 	    Caputo's formulation for the fractional order $\alpha$ is presented as:
 	    \begin{equation}\label{caputo}
 	           \fract{C}{\alpha}{}{}{D}h(x)=\frac{1}{\Gamma(v-\alpha)}\int_{a}^{x} \frac{h^{(v)}(t)dt}{(x-t)^{\alpha+1-v}}, \quad v-1<\alpha<v.
 	     \end{equation}
 	   \end{definition}

In Caputo’s derivative formula, the calculation of this derivative requires the calculation of an
integral, which may be complicated and computationally inefficient. Therefore, in the next theorem, a method
to approximate this derivative is introduced.

\begin{theorem}
    Let $0 < \alpha < 1$ and the interval $[x_0, x_m]$ is discretized to $m$ points, $0 = x_0 < x_1 < \dots < x_m$. By
subtracting an integral part in Caputo’s definition to $m$ sub-interval, Caputo’s fractional derivative is
defined as:
\begin{equation} \label{form_des}
    \fract{C}{\alpha}{}{}{D}h(x_m)=\frac{1}{\Gamma(m-\alpha)}\sum_{k=0}^{m-1}\int_{x_k}^{x_{k+1}} \frac{h^{\prime}(t)}{({x_m-t})^{\alpha}}dt.
\end{equation}

And since the first derivative of the function $h(x)$ is estimated as:
\begin{equation}
   h^{\prime}(t)\approx\frac{h_{k+1}-h_k}{f_{k+1}-f_k}, \quad x_k<t<x_{k+1},
   \end{equation}
that $h_k=h(f_k)$, then the formula (\ref{form_des}) can be expressed as:
\begin{equation}
    \fract{C}{\alpha}{}{}{D}h(x_m)\approx \sum_{k=0}^{m}(g_{k-1}-g_k)h_k,
\end{equation}
with $g_k$ defined as:
\begin{center}
 	    $g_k=\begin{cases}\frac{(x_m-x_k)^{1-\alpha}-(x_m-x_{k+1})^{1-\alpha}}{\Gamma(2-\alpha)(x_{k+1}-x_k}), & 1\leq k<m\\0, & otherwise\end{cases}$.
      \end{center}

 This process is known as $L1$ discretization, in which $\alpha$ is the order of the fractional derivative and assumes that the points $\Delta x=x_{k+1}-x_k$ are equidistant.
 \end{theorem}
\begin{proof}
    \cite{miller1993introduction}
\end{proof}

Fractional calculus has led to the development of different mathematical models, which have better consistency with experimental data in both physical and engineering problems, like fractional differential equations.

  \section{Methodology} \label{sec3}This section proposes a model for solving differential-algebraic equations using the least-squares support vector regression method in conjunction with the collocation approach and the Legendre polynomials.

 \subsection{Legendre Orthogonal Kernel}
 As a weighted residual approach, Spectral methods provide the advantage of choosing basis functions. In spectral methods, basis functions are expected to be infinitely differentiable and possess a completeness property. Commonly, orthogonal functions are employed as the basis, simplifying calculations and enabling the derivation of structured systems.   Within this context, Jacobi polynomials, a family of orthogonal polynomials defined over the interval $[-1,1]$ with respect to the weight function $w(x)$, include the renowned Legendre polynomial sequence. The Legendre polynomials serve as global functions, constituting a comprehensive basis for square-integrable functions on the interval $[-1,1]$. Their utilization leads to spectral convergence, enhancing solution accuracy as the degree of Legendre polynomials rises. Furthermore, explicit Legendre formulas enable efficient computation in numerical methods, making them an ideal choice for solving differential equations.

 	\begin{definition}
  Two polynomials from an infinite sequence of polynomials $\left\{P_n(x)\right\}^\infty_{n=0}$ are considered orthogonal if their inner product, determined for an integrable weight function $w(x)\geq0$ over the interval $[a,b]$, evaluates to zero. Furthermore, Legendre polynomials associated with the weight function $w(x)=1$ within the interval $[-1,1]$ are defined by the formula \cite{foupouagnigni2020orthogonal}:
\begin{equation} \label{leg_delta}
\langle{P_n(x),P_m(x)}\rangle=\frac{2}{2n+1}\delta_{nm},
\end{equation}
\end{definition}
where $\delta$ signifies the Kronecker delta function. The following formula presents the definition of the Legendre polynomial with degree $n$.
 	\begin{definition} These polynomials can be obtained through the expression \cite{shen2011spectral}:
 	\begin{equation}\label{legendr}
 	    P_n(x)=\sum_{v=0}^{\left[\frac{n}{2}\right]}(-1)^v\frac{(2n-2v)!}{2^n(n-v)!(n-2v)!v!}t^{n-2v},
 	\end{equation}
  where $v$ and $n$ are non-negative integers. For enhanced clarity, we calculate Legendre polynomials of degrees zero to three using (\ref{legendr}):
   \begin{align*}
      &P_0(x)=1\\
      &P_1(x)=x\\
      &P_2(x)=\frac{1}{2}(3x^2-1)\\
      &P_3(x)=\frac{1}{2}(5x^3-3x).\\
    \end{align*}
 	\end{definition}

Legendre polynomials can be transferred to arbitrary intervals with appropriate mappings that preserve the orthogonality property. In order to use shifted Legendre polynomials in the interval $[a,b]$, it is necessary to use an affine transformation. For this purpose, the linear mapping
  \begin{equation} \label{shifted}
      \phi(t)=\frac{2t-a-b}{b-a},
  \end{equation}
  that maps the interval $[-1,1]$ to the interval $[a,b]$ is applied.

	\subsection{LS-SVR for solving differential-algebraic equations}
In this section, we are trying to solve a system of differential-algebraic equations by using the least-squares support vector regression model and combining it with weighted residual methods. First, we begin by recalling the LS-SVR model in the context of regression problems.
\begin{theorem}
    Consider the training data $\left\{(x_k,y_k)\right\}^N_{k=1}$, where $x_k\in\mathbb{R}^d$ and $y\in\mathbb{R}$ are assumed. Within the framework of the least-squares support vector machine, the unknown function is formulated as follows:
    \begin{equation}
	    y(x)=\mathbf{w}^T\phi(x)+b,
	\end{equation}
 where $\mathbf{w}$ represents a weight vector, and $b$ corresponds to a bias term. The primal form of the optimization problem is expressed as follows:
 \begin{align*}
	\min\limits_{\mathbf{w}, e, b} \quad& \frac{1}{2} \mathbf{w}^T\mathbf{w} +\gamma \frac{1}{2}\sum_{k=1}^Ne_k^2 \\
	\text{ s.t.} \quad& y_k=\mathbf{w}^T\phi(x)+b+e_k, \quad k=1, \ldots , N,
	\end{align*}
 where, $e_k$ denotes the error variable, and $\gamma$ signifies the regularization parameter. Through the utilization of Lagrange multipliers and optimality conditions, the resultant system of linear equations for solving the regression problem is presented below:
 \begin{align*}
	\renewcommand\arraystretch{1.2}
	\left[
	\begin{array}{c|c}
	0 & 1_N^T \\\hline
	1_N & \Omega + I/\gamma
	\end{array}
	\right]
	\left[
	\begin{array}{c}
	b \\
	\boldsymbol{\alpha}
	\end{array}
	\right] =
	\left[
	\begin{array}{c}
	0 \\
	y
	\end{array}
	\right],
	\end{align*}
 here, $\Omega_{kl}=K(x_k,x_l)$ for all $k,l=1,\dots, N$, and $\boldsymbol{\alpha}=\left[\alpha_1,\alpha_2,\dots,\alpha_N\right]^T$ represents the Lagrange multipliers. Also, we have:
 \begin{equation}
	    y(x)=\sum_{k=1}^N\alpha_kK(x_k,x_l)+b,
	\end{equation}
 which the function $K(x_k,x_l)$ is the kernel function.
\end{theorem}
\begin{proof}
    \cite{suykens2002least}
\end{proof}
To solve the system of differential-algebraic equations, a set of $k$ equations is characterized by corresponding unknown functions $f_i$:
\begin{equation}
\mathcal{N}_i(u_1,u_2,\ldots,u_i) = f_i, \quad i=1,2,\ldots,k.
\label{eq:systemfull}
\end{equation}

The approximated solutions of $u_i$ can be expressed as a formulation derived from the research of Suykens et al. \cite{suykens2002least}:
\begin{equation*}
u_i\approx\tilde{u}_i(x) = w_i^T\varphi(x) + b_i=\sum_{j=0}^k w_j \varphi_j(x)+b_j, \quad j=1,2,\ldots,k,
\end{equation*}
where $\varphi_i(x)$ is replaced with Legendre functions of degree $i$, and $w_i$ and $b_i$ are the unknown coefficients.
In the next step, the unknown coefficients are set side by side in a vector form. Thus, we have:
\begin{equation*}
{\mathbf{w}} = [w_{1,1},w_{1,2},\ldots,w_{1,d}, w_{2,1},w_{2,2},\ldots,w_{2,d}, \ldots, w_{k,1},w_{k,2},\ldots,w_{k,d}],
\end{equation*}
as observed within $d$-dimensional vectors. In order to find the unknown coefficients $w_i$ and $b_i$ the optimization problem by using the inner product, which is defined in (\ref{leg_delta}), for solving \eqref{eq:systemfull} can be constructed as:
\begin{equation}
\begin{aligned}
\min\limits_{\mathbf{w}, \mathbf{e}} &\quad \frac{1}{2} \mathbf{w}^T\mathbf{w} + \frac{\gamma}{2}\mathbf{e}^T\mathbf{e} &\\
\text{ s.t.}& \quad \langle\mathcal{N}_i(\tilde{u}_1,\tilde{u}_2,\ldots,\tilde{u}_k) - f_i,\psi_j\rangle =e_{i,j}, & j=1, \ldots, n,
\end{aligned}
\label{eq:lssvrsystem1}
\end{equation}
where $i=1,2,\ldots,k$ shows that $kn$ is the number of training points, $e_{i,j}$ is the residual value, and $\psi_j$ is a set of basis functions in the test space. The matrix $e_{i,j}$ should be vectorized the same as unknown coefficients $\mathbf{w}$. For any linear operator $\mathcal{N}$, denoted by $\mathcal{L}$, the dual form of the optimization problem \eqref{eq:lssvrsystem1} can be derived. Here we obtain the dual form for a system of three equations and three unknown functions. This process can be generalized for an arbitrary number of equations. Suppose the following system of equations is given:
\begin{equation}
\begin{cases}
\mathcal{L}_1(u_1,u_2,u_3) = f_1\\
\mathcal{L}_2(u_1,u_2,u_3) = f_2\\
\mathcal{L}_3(u_1,u_2,u_3) = f_3
\end{cases}.
\end{equation}

By approximating the solutions using,
\begin{equation}\label{eq:solutions}
\begin{aligned}
\tilde{u}_1(x) &= w_1^T\varphi(x) + b_1. \\
\tilde{u}_2(x) &= w_2^T\varphi(x) + b_2. \\
\tilde{u}_3(x) &= w_3^T\varphi(x) + b_3. \\
\end{aligned}
\end{equation}

Note that, for partial differential-algebraic equations, $\tilde{u}_i$ is defined as $\tilde{u}_i(x,t)= w_{ i,j}{\varphi}^T_j(x)\varphi_i(t)$, which applies to the general form we present. The optimization problem \eqref{eq:lssvrsystem1} takes the form
\begin{equation*}
\begin{aligned}
\min\limits_{\mathbf{w},\mathbf{e}} &\quad \frac{1}{2} \mathbf{w}^T\mathbf{w} + \frac{\gamma}{2}\mathbf{e}^T\mathbf{e} &\\
\text{ s.t.}& \quad \langle\mathcal{L}_1(\tilde{u}_1,\tilde{u}_2,\tilde{u}_3) - f_1,\psi_j\rangle =e_{1,j}, & j&=1, \ldots, n, \\
\text{ s.t.}& \quad
\langle\mathcal{L}_2(\tilde{u}_1,\tilde{u}_2,\tilde{u}_3) - f_2,\psi_j\rangle =e_{2,j}, & j&=n+1, \ldots, 2n, \\
\text{ s.t.}& \quad \langle\mathcal{L}_3(\tilde{u}_1,\tilde{u}_2,\tilde{u}_3) - f_3,\psi_j\rangle =e_{3,j}, & j&=2n+1, \ldots, 3n,
\end{aligned}
\end{equation*}
where
\begin{equation}
\begin{aligned}
\mathbf{w} &= [w_1,w_2,w_3] = [w_{1,1}, w_{1,2},\ldots,w_{1,n},w_{2,1}, w_{2,2},\ldots,w_{2,n},w_{3,1},w_{3,2},\ldots,w_{3,n}]\\
\mathbf{e} &= [e_1,e_2,e_3] = [e_{1,1}, e_{1,2},\ldots,e_{1,n},e_{2,1}, e_{2,2},\ldots,e_{2,n},e_{3,1}, e_{3,2},\ldots,e_{3,n}].
\end{aligned}
\end{equation}

In order to solve this system, the Lagrangian function should be constructed as follows:
\begin{equation*}
\begin{aligned}
\mathfrak{L}(\mathbf{w},\mathbf{e},\boldsymbol{\alpha}) &= \frac{1}{2} \mathbf{w}^T\mathbf{w} + \frac{\gamma}{2}\mathbf{e}^T\mathbf{e} - \sum_{j=1}^n \alpha_j \langle\mathcal{L}_1(\tilde{u}_1,\tilde{u}_2,\tilde{u}_3) - f_1,\psi_j\rangle -e_{j} -\sum_{j=1}^n \alpha_{n+j} \langle\mathcal{L}_2(\tilde{u}_1,\tilde{u}_2,\tilde{u}_3) - f_2,\psi_j\rangle -e_{n+j} \\ &-\sum_{j=1}^n \alpha_{2n+j} \langle\mathcal{L}_3(\tilde{u}_1,\tilde{u}_2,\tilde{u}_3) - f_3,\psi_j\rangle -e_{2n+j}.
\end{aligned}
\end{equation*}

The conditions for optimality of the Lagrangian function are partial derivatives $\frac{\partial \mathfrak{L}}{\partial w_k}, \frac{\partial \mathfrak{L}}{\partial e_k}, \frac{\partial \mathfrak{L}}{\partial b}$, and $\frac{\partial \mathfrak{L}}{\partial \alpha_k}$ which should be equaled zero. Then we compute these equations to find the coefficients:

\begin{equation}
\begin{aligned}
\frac{\partial \mathfrak{L}}{\partial w_k} = 0\rightarrow w_k =
\begin{cases}
\displaystyle \sum_{j=1}^n \alpha_j \langle\mathcal{L}_1(\varphi_k,0,0) - f_1,\psi_j\rangle -e_{j} + \sum_{j=1}^n \alpha_{n+j} \langle\mathcal{L}_2(\varphi_k,0,0) - f_2,\psi_j\rangle -e_{n+j} + \\ \displaystyle \sum_{j=1}^n \alpha_{2n+j} \langle\mathcal{L}_3(\varphi_k,0,0) - f_3,\psi_j\rangle -e_{2n+j}, & k=1,2,\ldots,d. \\
\displaystyle \sum_{j=1}^n \alpha_j \langle\mathcal{L}_1(0,\varphi_k,0) - f_1,\psi_j\rangle -e_{j} + \sum_{j=1}^n \alpha_{n+j} \langle\mathcal{L}_2(0,\varphi_k,0) - f_2,\psi_j\rangle -e_{n+j} + \\ \displaystyle \sum_{j=1}^n \alpha_{2n+j} \langle\mathcal{L}_3(0,\varphi_k,0) - f_3,\psi_j\rangle -e_{2n+j}, & k=1,2,\ldots,d. \\
\displaystyle \sum_{j=1}^n \alpha_j \langle\mathcal{L}_1(0,0,\varphi_k) - f_1,\psi_j\rangle -e_{j} + \sum_{j=1}^n \alpha_{n+j} \langle\mathcal{L}_2(0,0,\varphi_k) - f_2,\psi_j\rangle -e_{n+j} + \\ \displaystyle \sum_{j=1}^n \alpha_{2n+j} \langle\mathcal{L}_3(0,0,\varphi_k) - f_3,\psi_j\rangle -e_{2n+j}, & k=1,2,\ldots,d. \\
\end{cases}\\
\end{aligned}
\end{equation}
\begin{equation}
\begin{aligned}
\frac{\partial \mathfrak{L}}{\partial e_k} &= 0 \rightarrow \gamma e_k + \alpha_k = 0, \qquad k=1\ldots 3n.\\
\frac{\partial \mathfrak{L}}{\partial b} &= 0\rightarrow
\begin{cases}
\displaystyle \frac{\partial \mathfrak{L}}{\partial b_1} &= 0 \rightarrow \displaystyle \sum_{i=1}^n \alpha_i\langle\mathcal{L}_1(1,0,0),\psi_i\rangle + \sum_{i=1}^n \alpha_{n+i}\langle\mathcal{L}_2(1,0,0),\psi_i\rangle,
\sum_{i=1}^n \alpha_{2n+i}\langle\mathcal{L}_3(1,0,0),\psi_i\rangle,\\
\displaystyle \frac{\partial \mathfrak{L}}{\partial b_2} &= 0 \rightarrow \displaystyle \sum_{i=1}^n \alpha_i\langle\mathcal{L}_1(0,1,0),\psi_i\rangle + \sum_{i=1}^n \alpha_{n+i}\langle\mathcal{L}_2(0,1,0),\psi_i\rangle,
\sum_{i=1}^n \alpha_{2n+i}\langle\mathcal{L}_3(0,1,0),\psi_i\rangle,\\\displaystyle \frac{\partial \mathfrak{L}}{\partial b_3} &= 0 \rightarrow \displaystyle \sum_{i=1}^n \alpha_i\langle\mathcal{L}_1(0,0,1),\psi_i\rangle + \sum_{i=1}^n \alpha_{n+i}\langle\mathcal{L}_2(0,0,1),\psi_i\rangle,
\sum_{i=1}^n \alpha_{2n+i}\langle\mathcal{L}_3(0,0,1),\psi_i\rangle,\\
\end{cases}\\
\frac{\partial \mathfrak{L}}{\partial \alpha_k} &= 0\rightarrow
\begin{cases}
\displaystyle \sum_{j=1}^d w_j \langle\mathcal{L}_1(\varphi_j,0,0) - f_1,\psi_k\rangle + \sum_{j=1}^d w_{d+j} \langle\mathcal{L}_1(0,\varphi_j,0) - f_1,\psi_k\rangle + \\
\displaystyle \sum_{j=1}^d w_{2d+j} \langle\mathcal{L}_1(0,0,\varphi_j) - f_1,\psi_k\rangle =e_{k} & k=1,2,\ldots,n.  \\

\displaystyle \sum_{j=1}^d w_j \langle\mathcal{L}_2(\varphi_j,0,0) - f_2,\psi_k\rangle + \sum_{j=1}^d w_{d+j} \langle\mathcal{L}_2(0,\varphi_j,0) - f_2,\psi_k\rangle + \\
\displaystyle \sum_{j=1}^d w_{2d+j} \langle\mathcal{L}_2(0,0,\varphi_j) - f_2,\psi_k\rangle =e_{k} & k=n+1,n+2,\ldots,2n.  \\

\displaystyle \sum_{j=1}^d w_j \langle\mathcal{L}_3(\varphi_j,0,0) - f_3,\psi_k\rangle + \sum_{j=1}^d w_{d+j} \langle\mathcal{L}_3(0,\varphi_j,0) - f_3,\psi_k\rangle + \\
\displaystyle \sum_{j=1}^d w_{2d+j} \langle\mathcal{L}_3(0,0,\varphi_j) - f_3,\psi_k\rangle =e_{k} & k=2n+1,2n+2,\ldots,3n.  \\
\end{cases}\\
\end{aligned}
\label{eq:optsystem}
\end{equation}

By defining
\begin{equation*}
\begin{aligned}
A^{(1)}_{i,j,k} & = \langle\mathcal{L}_1(\varphi_i,0,0),\psi_j\rangle,
A^{(2)}_{i,j,k}= \langle\mathcal{L}_2(\varphi_i,0,0),\psi_j\rangle,
A^{(3)}_{i,j,k} = \langle\mathcal{L}_3(\varphi_i,0,0),\psi_j\rangle,
\\
A^{(4)}_{i,j,k} & = \langle\mathcal{L}_1(0,\varphi_i,0),\psi_j\rangle,
A^{(5)}_{i,j,k} = \langle\mathcal{L}_2(0,\varphi_i,0),\psi_j\rangle,
A^{(6)}_{i,j,k} = \langle\mathcal{L}_3(0,\varphi_i,0),\psi_j\rangle,
\\
A^{(7)}_{i,j,k}& = \langle\mathcal{L}_1(0,0,\varphi_i),\psi_j\rangle,
A^{(8)}_{i,j,k} = \langle\mathcal{L}_2(0,0,\varphi_i),\psi_j\rangle,
A^{(9)}_{i,j,k} = \langle\mathcal{L}_3(0,0,\varphi_i),\psi_j\rangle,
\\
B^{(1)}_{i,j,k}& = \langle\mathcal{L}_1(1,0,0),\psi_j\rangle,
B^{(2)}_{i,j,k} = \langle\mathcal{L}_2(1,0,0),\psi_j\rangle,
B^{(3)}_{i,j,k} = \langle\mathcal{L}_3(1,0,0),\psi_j\rangle,
\\
B^{(4)}_{i,j,k}& = \langle\mathcal{L}_1(0,1,0),\psi_j\rangle,
B^{(5)}_{i,j,k} = \langle\mathcal{L}_2(0,1,0),\psi_j\rangle,
B^{(6)}_{i,j,k} = \langle\mathcal{L}_3(0,1,0),\psi_j\rangle,
\\
B^{(7)}_{i,j,k}& = \langle\mathcal{L}_1(0,0,1),\psi_j\rangle,
B^{(8)}_{i,j,k} = \langle\mathcal{L}_2(0,0,1),\psi_j\rangle,
B^{(9)}_{i,j,k} = \langle\mathcal{L}_3(0,0,1),\psi_j\rangle,
\\
\end{aligned}
\end{equation*}
and
\begin{equation*}
\begin{aligned}
Z &=\begin{bmatrix}A^{(1)} & A^{(2)} & A^{(3)} \\A^{(4)} & A^{(5)} & A^{(6)} \\A^{(7)} & A^{(8)} & A^{(9)}\end{bmatrix},\\
V &=\begin{bmatrix}B^{(1)} & B^{(2)} & B^{(3)} \\B^{(4)} & B^{(5)} & B^{(6)} \\B^{(7)} & B^{(8)} & B^{(9)}\end{bmatrix},\\
\end{aligned}
\end{equation*}
the relation \eqref{eq:optsystem} can be reformulated as:
\begin{equation*}
\begin{aligned}
\begin{cases}

Z \boldsymbol{\alpha}= \mathbf{w}\\

\mathbf{e} = -\boldsymbol{\alpha}/\gamma\\

b=V^T\boldsymbol{\alpha}\\

Z^T \mathbf{w} - \mathbf{e} = y.
\end{cases}
\end{aligned}
\end{equation*}

Eliminating $\mathbf{w}$ and $\mathbf{e}$ yields
\begin{equation}
\begin{aligned}
\renewcommand\arraystretch{1.5}
\left[
\begin{array}{c|c}
\mathbf{0} & V^T \\\hline
V & \Omega + I/\gamma
\end{array}
\right]
\left[
\begin{array}{c}
b \\ \hline
\boldsymbol{\alpha}
\end{array}
\right] =
\left[
\begin{array}{c}
0 \\\hline
y
\end{array}
\right],
\end{aligned}
\end{equation}
where
\begin{equation*}
\begin{aligned}
\boldsymbol{\alpha} &= \left[\alpha_1,\ldots,\alpha_{3n}\right]^T,\\
y& = \left[\langle f_1, \psi_1 \rangle,\langle f_1, \psi_2 \rangle,\ldots,\langle f_1, \psi_n \rangle, \langle f_2, \psi_1 \rangle,\langle f_2, \psi_2 \rangle,\ldots,\langle f_2, \psi_n \rangle,
\langle f_3, \psi_1 \rangle,\langle f_3, \psi_2 \rangle,\ldots,\langle f_3, \psi_n \rangle \right]^T,\\
\end{aligned}
\end{equation*}
and
\begin{equation}
\Omega = Z^T Z = \begin{bmatrix}A^T & C^T \\B^T & D^T \end{bmatrix}
\begin{bmatrix}A & B \\C & D \end{bmatrix}.
\end{equation}

The kernel trick also appears at each block of matrix $\Omega$. As a result, the approximated solution in the dual form can be computed using:
\begin{equation*}
\begin{aligned}
\tilde{u}_1 (x) &= \sum_{i=1}^n \alpha_i
\widetilde{K}_1(x,x_i) + \sum_{i=1}^n \alpha_i
\widetilde{K}_2(x,x_i) + \sum_{i=1}^n \alpha_i
\widetilde{K}_3(x,x_i) +b_1, \\
\tilde{u}_2 (x) &= \sum_{i=1}^n \alpha_i
\widetilde{K}_4(x,x_i) + \sum_{i=1}^n \alpha_i
\widetilde{K}_5(x,x_i) \sum_{i=1}^n \alpha_i
\widetilde{K}_6(x,x_i)+ b_2,\\
\tilde{u}_3 (x) &= \sum_{i=1}^n \alpha_i
\widetilde{K}_7(x,x_i) + \sum_{i=1}^n \alpha_i
\widetilde{K}_8(x,x_i) \sum_{i=1}^n \alpha_i
\widetilde{K}_9(x,x_i) +b_3,
\end{aligned}
\end{equation*}
where
\begin{equation*}
\begin{aligned}
\widetilde{K}_1(x,x_i) = \langle \langle \mathcal{L}_1(\varphi,0,0), \psi_i \rangle , \varphi\rangle = \langle\mathcal{L}_1(K(x,t),0,0),\psi_i\rangle,\\
\widetilde{K}_2(x,x_i) = \langle \langle \mathcal{L}_2(\varphi,0,0), \psi_i \rangle , \varphi\rangle= \langle\mathcal{L}_2(K(x,t),0,0),\psi_i\rangle,\\
\widetilde{K}_3(x,x_i) = \langle \langle \mathcal{L}_3(\varphi,0,0), \psi_i \rangle , \varphi\rangle = \langle\mathcal{L}_3(K(x,t),0,0),\psi_i\rangle,\\
\widetilde{K}_4(x,x_i) = \langle \langle \mathcal{L}_1(0,\varphi,0), \psi_i \rangle , \varphi\rangle = \langle\mathcal{L}_1(0,K(x,t),0),\psi_i\rangle,\\
\widetilde{K}_5(x,x_i) = \langle \langle \mathcal{L}_2(0,\varphi,0), \psi_i \rangle , \varphi\rangle= \langle\mathcal{L}_2(0,K(x,t),0),\psi_i\rangle,\\
\widetilde{K}_6(x,x_i) = \langle \langle \mathcal{L}_3(0,\varphi,0), \psi_i \rangle , \varphi\rangle = \langle\mathcal{L}_3(0,K(x,t),0),\psi_i\rangle,\\
\widetilde{K}_7(x,x_i) = \langle \langle \mathcal{L}_1(0,0,\varphi), \psi_i \rangle , \varphi\rangle = \langle\mathcal{L}_1(0,0,K(x,t)),\psi_i\rangle,\\
\widetilde{K}_8(x,x_i) = \langle \langle \mathcal{L}_2(0,0,\varphi), \psi_i \rangle , \varphi\rangle= \langle\mathcal{L}_2(0,0,K(x,t)),\psi_i\rangle,\\
\widetilde{K}_9(x,x_i) = \langle \langle \mathcal{L}_3(0,0,\varphi), \psi_i \rangle , \varphi\rangle = \langle\mathcal{L}_3(0,0,K(x,t)),\psi_i\rangle.\\
\end{aligned}
\end{equation*}

 Hence, the proposed approach formulates the solutions (\ref{eq:solutions}) using kernel functions $(K(x,t))$ which are replaced by Legendre polynomials \cite{parand2021new}. Also, it would be possible to employ $\boldsymbol{\alpha}$ to calculate $\mathbf{w}$, and by utilizing Formula (\ref{eq:solutions}), one could approximate the solution.

	\begin{remark}

	  Although the choice of basis functions can differ for each approximated function, for simplicity, uniform functions are employed. In this proposed methodology, Legendre polynomials (\ref{legendr}) of degree one, denoted as $P_1(x)=x$, are utilized. Additionally, considering that $P_0(x) = 1$, the variable $b$ (bias term) is omitted from this formulation.
\end{remark}
\begin{remark}
    Since the Legendre function is defined in the interval $(-1,1)$ in the formula (\ref{leg_delta}), we built the kernel of the network by the shifted Legendre polynomials using (\ref{shifted}).
\end{remark}

\begin{remark}
    The collocation method considers the test function as $\psi_j(x)=\delta(x-x_j)$, where $\delta$ is the Dirac delta function. According to the characteristic of Dirac’s Delta function, in the formula (\ref{eq:lssvrsystem1}) we have $\langle\mathcal{N}_i(\tilde{u}_1,\tilde{u}_2,\ldots,\tilde{u}_k) - f_i,\delta(x-x_j)\rangle =\mathcal{N}_i(\tilde{u}_1(x_j),\tilde{u}_2(x_j),\ldots,\tilde{u}_k(x_j)) - f_i(x_j)$.
\end{remark}
This methodology is termed Collocation Least-Squares Support Vector Regression (CLS-SVR). In this algorithm, we employ the roots of Legendre polynomials of degree $m$ as the training data for the CLS-SVR learning process.

	\section{Numerical results} \label{sec4}
	To show the efficiency of the proposed method, this section presents a numerical solution to various well-known differential-algebraic equations. This example includes DAEs, VI-DAE of fractional order, linear and nonlinear FDAEs, and a system of partial differential-algebraic equations. All the results have been obtained by running the algorithm using Maple 20. The error criterion of the presented tables is the relative error defined by:
 \begin{equation}
 \mathbb{E}_u=|\frac{u(t)-\tilde{u}(t)}{u(t)}|,
 \end{equation}
  where $u(t)$ is the exact solution and $\tilde{u}(t)$ is the predicted answer of the algorithm. In all these examples, the value of the regularization coefficient $\gamma$ is chosen between $10$ and $10^3$. These differences are caused by the structure of the problem and the limitations of Maple. In addition to the numerical tables and relative error for each example, the graph of the absolute error is given.

	\begin{example}\label{ex1}
	As the first example, consider the following differential-algebraic equation
		\begin{figure}\label{12}
		\centering
		\begin{subfigure}{.32\textwidth}
			\centering
\includegraphics[width=1\linewidth]{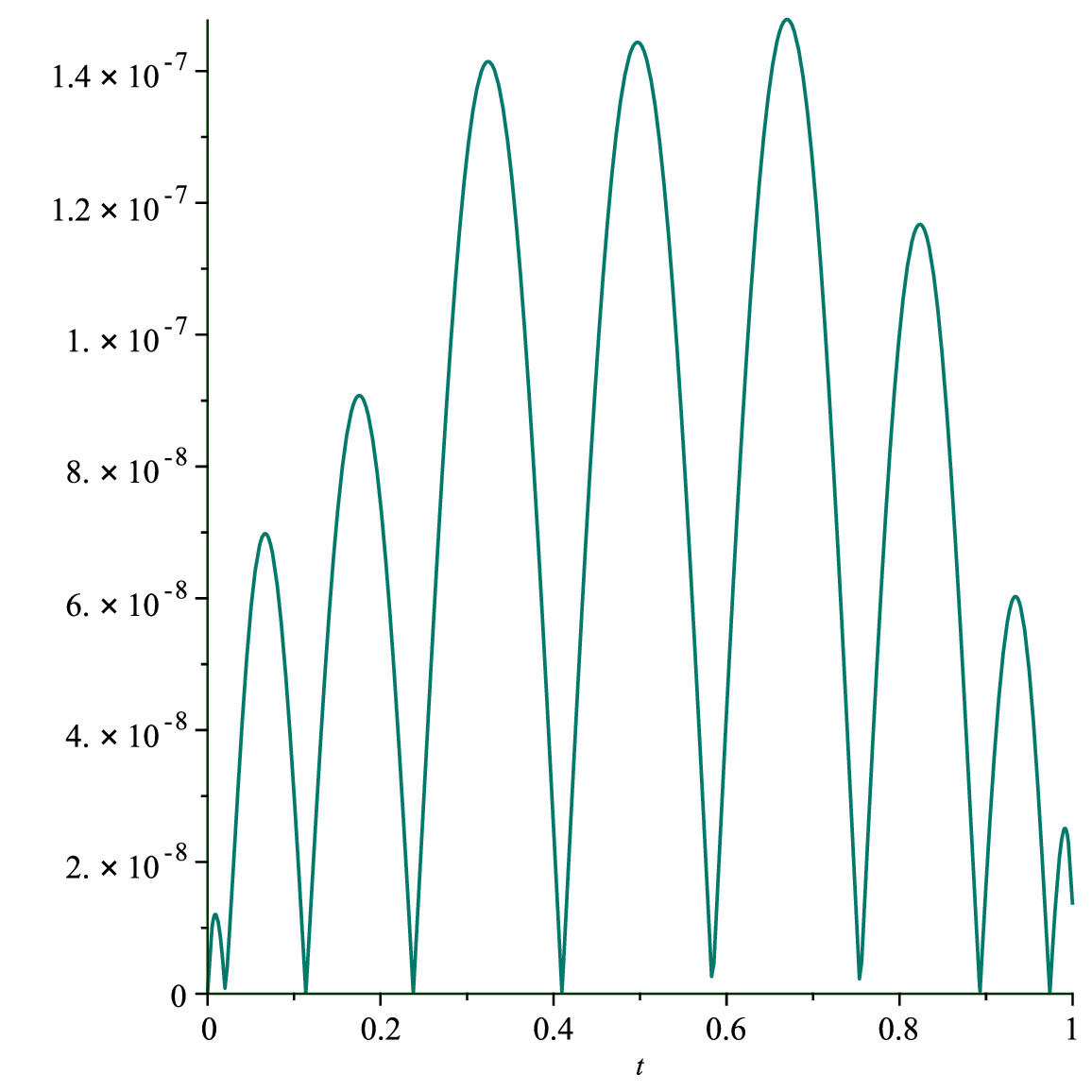}
			\caption{}
			\label{img:1}
		\end{subfigure}
		\begin{subfigure}{.32\textwidth}
			\centering
			\includegraphics[width=1\linewidth]{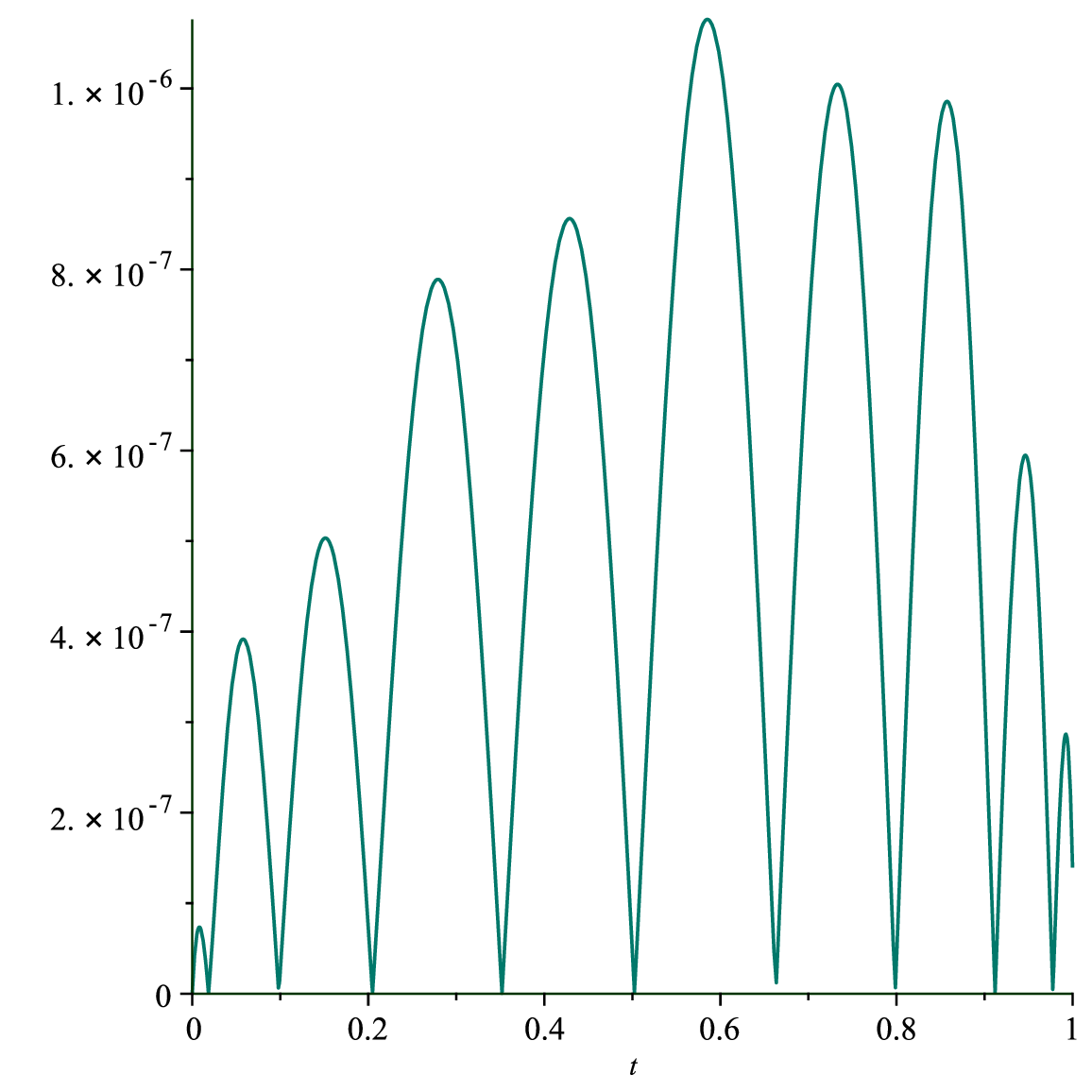}
			\caption{}
			\label{img:2}
		\end{subfigure}
   \begin{subfigure}
        {.32\textwidth}
			\centering
			\includegraphics[width=1\linewidth]{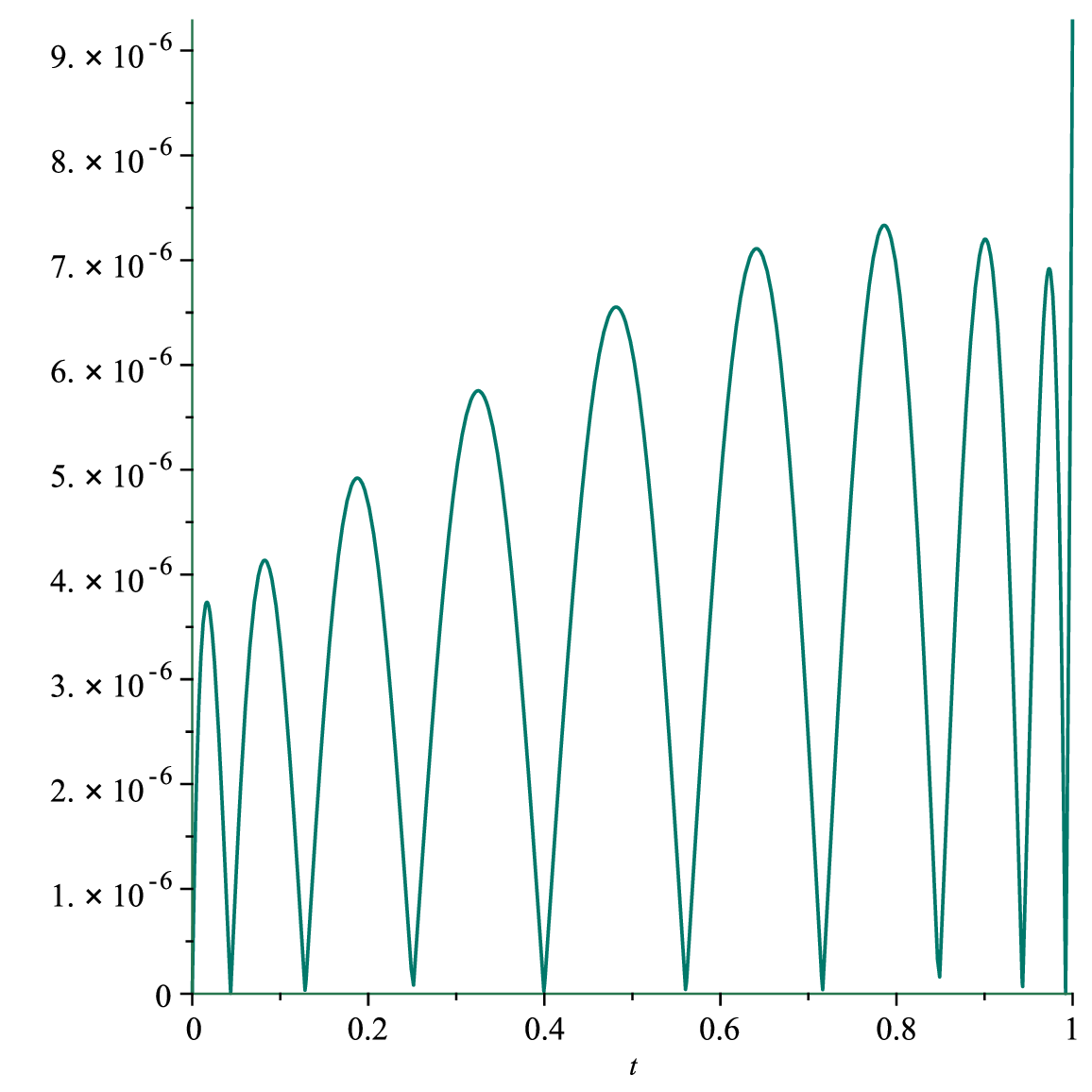}
			\caption{}

		\end{subfigure}
		\caption{(a) $|u_1(t)-\tilde{u}_1(t)|$, (b) $|u_2(t)-\tilde{u}_2(t)|$, and (c) $|u_3(t)-\tilde{u}_3(t)|$ for example \ref{ex1} with $m=4$.}
	\end{figure}\label{fig1}

\begin{table*}[]
\centering
 \captionsetup{font=scriptsize}
\caption{Exact, approximate solution, and relative error for $m=10$ on $\left[0,1\right]$ for example \ref{ex1}}.
\small
\resizebox{0.9\textwidth}{!}{
 \begin{tabular}
{|p{0.40\linewidth}|p{0.35\linewidth}|p{0.33\linewidth}|}

\begin{tabular}{@{}llll@{}}
\toprule
$t$ & $u_1(t)$ & $\Tilde{u}_1(t)$ & $\mathbb{E}_{u_1(t)}$\\ \midrule

{0.2} & 0.0397338& 0.0397337& $1.8 \times10^{-6}$\\
{0.4} & 0.15576733 & 0.15576736 & $1.6 \times10^{-7}$  \\
{0.6} & 0.3387854 & 0.3387855 & $1.2 \times10^{-7}$\\
{0.8} & 0.5738848 & 0.5738847 & $1.7 \times10^{-7}$ \\
{1} & 0.84147098 & 0.84147097  & $1.6 \times10^{-8}$ \\
\bottomrule
\end{tabular}

&

\begin{tabular}{@{}lll@{}}
\toprule
 $u_2(t)$ & $\Tilde{u}_2(t)$ & $\mathbb{E}_{u_2(t)}$\\ \midrule
 0.20271 &0.20270 & $3.5 \times10^{-7}$\\
 0.422793 & 0.422792 & $1.6 \times10^{-6}$ \\
0.684136 & 0.684137 & $1.5 \times10^{-6}$   \\
1.02963855 & 1.02963859 & $3.6 \times10^{-8}$ \\
 1.5574077 & 1.5574078 & $8.9 \times10^{-8}$  \\
\bottomrule
\end{tabular}
&
\begin{tabular}{@{}lll@{}}
\toprule
 $u_3(t)$ & $\Tilde{u}_3(t)$ & $\mathbb{E}_{u_3(t)}$\\ \midrule
 0.196013 & 0.196017 & $2.3 \times10^{-5}$ \\
0.3684243 & 0.3684244 &$1.2 \times10^{-7}$ \\
 0.495201& 0.495196 & $9.8 \times10^{-6}$    \\
 0.55736 & 0.55737 & $1.2 \times10^{-5}$ \\
0.540302 & 0.540293 & $1.7\times10^{-5}$ \\
\bottomrule
\end{tabular} \\
\hline
\end{tabular}}
\end{table*}\label{tab1}

\begin{equation}
\begin{cases}

u^{\prime}_1(t)=u_1(t)-u_3(t)u_2(t)+sin(t)+tcos(t),
\\

u^{\prime}_2(t)=tu_3(t)+u^2_1(t)+sec^2(t)-t^2(cos(t)+sin^2(t)),
\\

0=u_1(t)-u_3(t)+t(cos(t)-sin(t)), \quad t\in [0,1]
\\
\end{cases}
\end{equation}
	with initial conditions $u_1(0)=u_2(0)=u_3(0)=0$ and the exact solutions as $u_1(t)=tsin(t), u_2(t)=tan(t)$ and $u_3(t)=tcos(t)$. For this case, we set $m=10$ and $\gamma=10^2$ over the interval $\left[0,1\right]$. The relative error value obtained via the CLS-SVR method for each $u_i$ is denoted as $\mathbb{E}_{u_i}$.
  In Table \ref{tab1}, the error of this case is presented. In Figure \ref{fig1}, the absolute error of the $u_1(t)$, $u_2(t)$, and $u_3(t)$ is plotted.
\end{example}

\begin{example}\label{ex2}
In this example, we consider a fractional order
integro-differential algebraic equation :
	\begin{equation}
\begin{cases}

D^{\frac{1}{2}}u_1(t)=+t\int_{0}^{t} u_1(s)ds+(1+t)\int_{0}^{t}u_2(s)ds +(\frac{3t\sqrt{\pi}}{4}-\frac{2t^{\frac{7}{2}}}{5}-\frac{2(1+t)t^{\frac{5}{2}}}{5}) ,
\\
0=\int_{0}^{t} (1+s)u_1(s)ds +\int_{0}^{t} u_2(s)ds + (\frac{2t^{\frac{5}{2}}(5t+7)}{35}+\frac{2t^{\frac{5}{2}}}{5}), \quad t\in[0,1],
\end{cases}
\end{equation}
with zero initial conditions and the exact solutions $u_1(t) = u_2(t) = t\sqrt{t}$.

\begin{table*}[]
\centering
 \captionsetup{font=scriptsize}
\caption{Exact, approximate solution, and relative error for $m=14$ on $\left[0,1\right]$ for example \ref{ex2}}.
\small
 \begin{tabular}
{|p{0.39\linewidth}|p{0.30\linewidth}|}

\begin{tabular}{@{}llll@{}}
\toprule
$t$ & $u_1(t)$ & $\Tilde{u}_1(t)$ & $\mathbb{E}_{u_1(t)}$\\ \midrule

{0.2} & 0.089442 & 0.089452& $1.1 \times10^{-4}$  \\
{0.4} & 0.252982 & 0.252989 & $2.8 \times10^{-5}$   \\
{0.6} & 0.464758 & 0.464760 & $5.5 \times10^{-6}$    \\
{0.8} & 0.715541 & 0.715542& $1.5 \times10^{-6}$  \\
{1} & 1 & 1.0000007        & $7.3 \times10^{-7}$  \\
\bottomrule
\end{tabular}

&

\begin{tabular}{@{}lll@{}}
\toprule
 $u_2(t)$ & $\Tilde{u}_2(t)$ & $\mathbb{E}_{u_2(t)}$\\ \midrule
   0.089442 & 0.089690 & $2.7 \times10^{-3}$ \\
 0.252982 & 0.253186 & $8.0 \times10^{-4}$  \\
0.464758 & 0.465094 & $7.3 \times10^{-4}$   \\
 0.715541 & 0.715935& $5.5 \times10^{-4}$ \\
 1 & 1.01319  & $1.3 \times10^{-2}$\\
\bottomrule
\end{tabular}
\\
\hline
\end{tabular}
\end{table*}
\label{tab2}

	\begin{figure}
		\centering
		\begin{subfigure}{.4\textwidth}
			\centering			\includegraphics[width=1\linewidth]{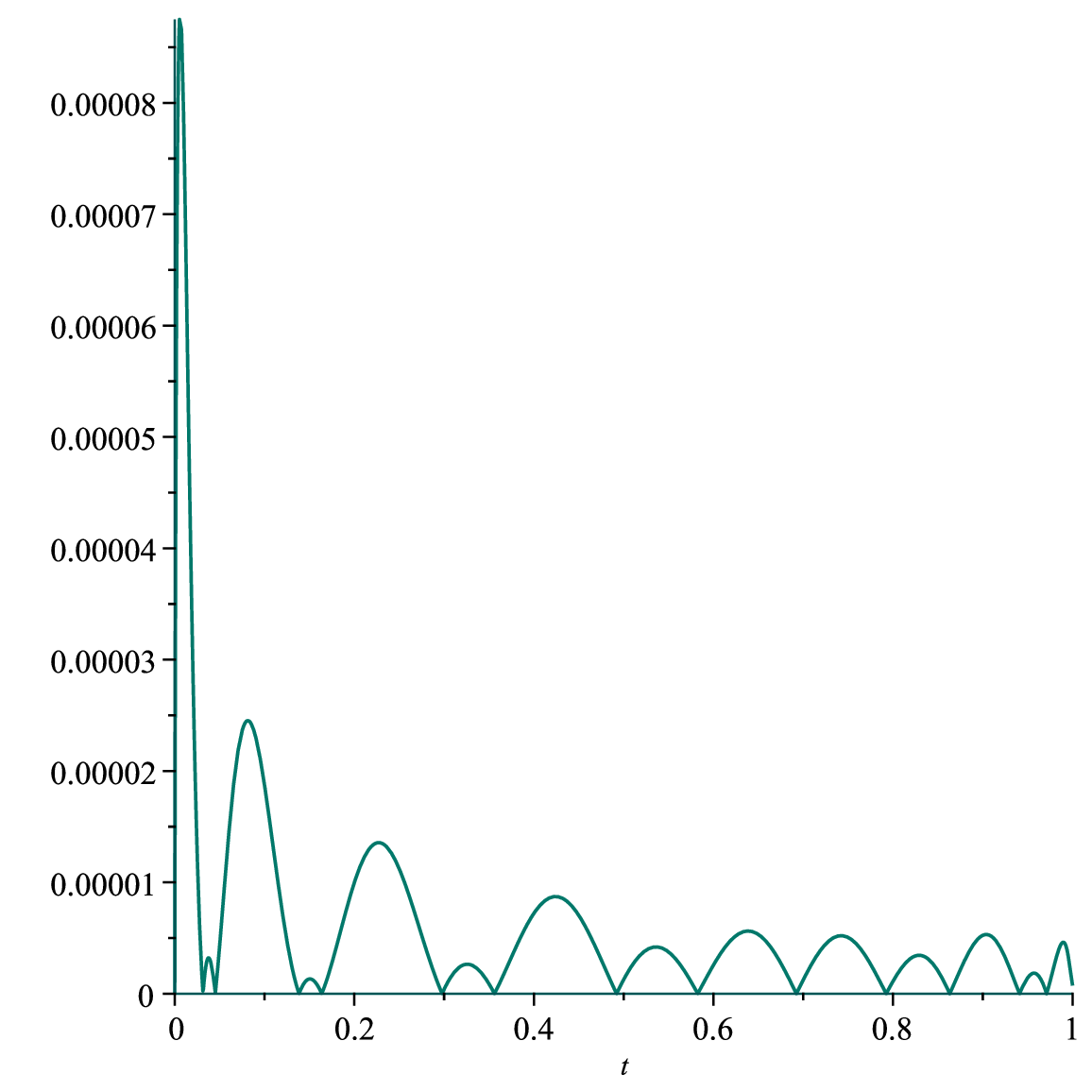}
			\caption{}

		\end{subfigure}
		\begin{subfigure}{.4\textwidth}
			\centering
			\includegraphics[width=1\linewidth]{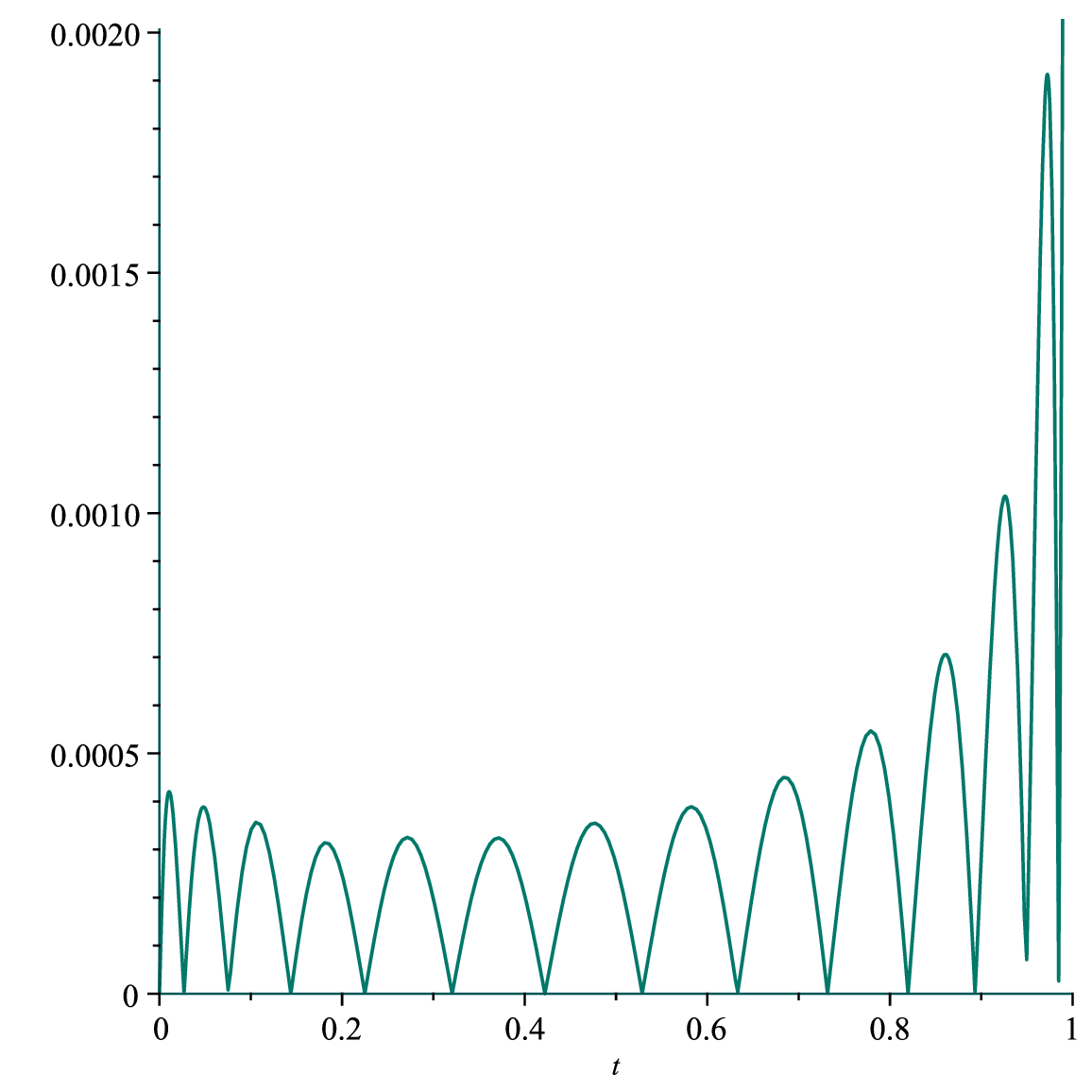}
			\caption{}

		\end{subfigure}
		\caption{(a) $|u_1(t)-\tilde{u}_1(t)|$ and (b) $|u_2(t)-\tilde{u}_2(t)|$ for example \ref{ex2} with $m=14$.}\label{fig2}
	\end{figure}

The provided equations belong to the category of Volterra integro-differential algebraic equations of fractional order. This study employs the Caputo derivative (\ref{caputo}) to account for fractional derivatives in the differential equations. To address this equation, the proposed model is trained using parameters $m=14$ and $\gamma=10^3$ across the interval $[0,1]$. Details of the error analysis resulting from the program's execution are provided in Table \ref{tab2} and Figure \ref{fig2}.
 \end{example}

    \begin{example}\label{ex3}
       In the third problem, we calculate the solution of a linear fractional differential-algebraic equation

\begin{equation}
\begin{cases}

D^{\frac{1}{2}}u_1(t)+2u_1(t)-\frac{\Gamma(\frac{7}{2})}{\Gamma(3)}u_2(t)+u_3(t)=2t^{\frac{5}{2}}+sin(t),
\\

D^{\frac{1}{2}}u_2(t)+u_2(t)+u_3(t)=\frac{\Gamma(3)}{\Gamma(\frac{5}{2})}t^{\frac{3}{2}}+t^2+sin(t),
\\

2u_1(t)+u_2(t)-u_3(t)=2t^{\frac{5}{2}}+t^2-sin(t), \quad t\in[0,1]
\\
\end{cases}
\end{equation}
in the context of zero initial conditions. The corresponding exact solutions are given by $u_1(t) = t^{\frac{5}{2}}$, $u_2(t) = t^2$, and $u_3(t) = \sin(t)$.
\end{example}

      \begin{figure}
		\centering
		\begin{subfigure}{.32\textwidth}
			\centering			\includegraphics[width=1\linewidth]{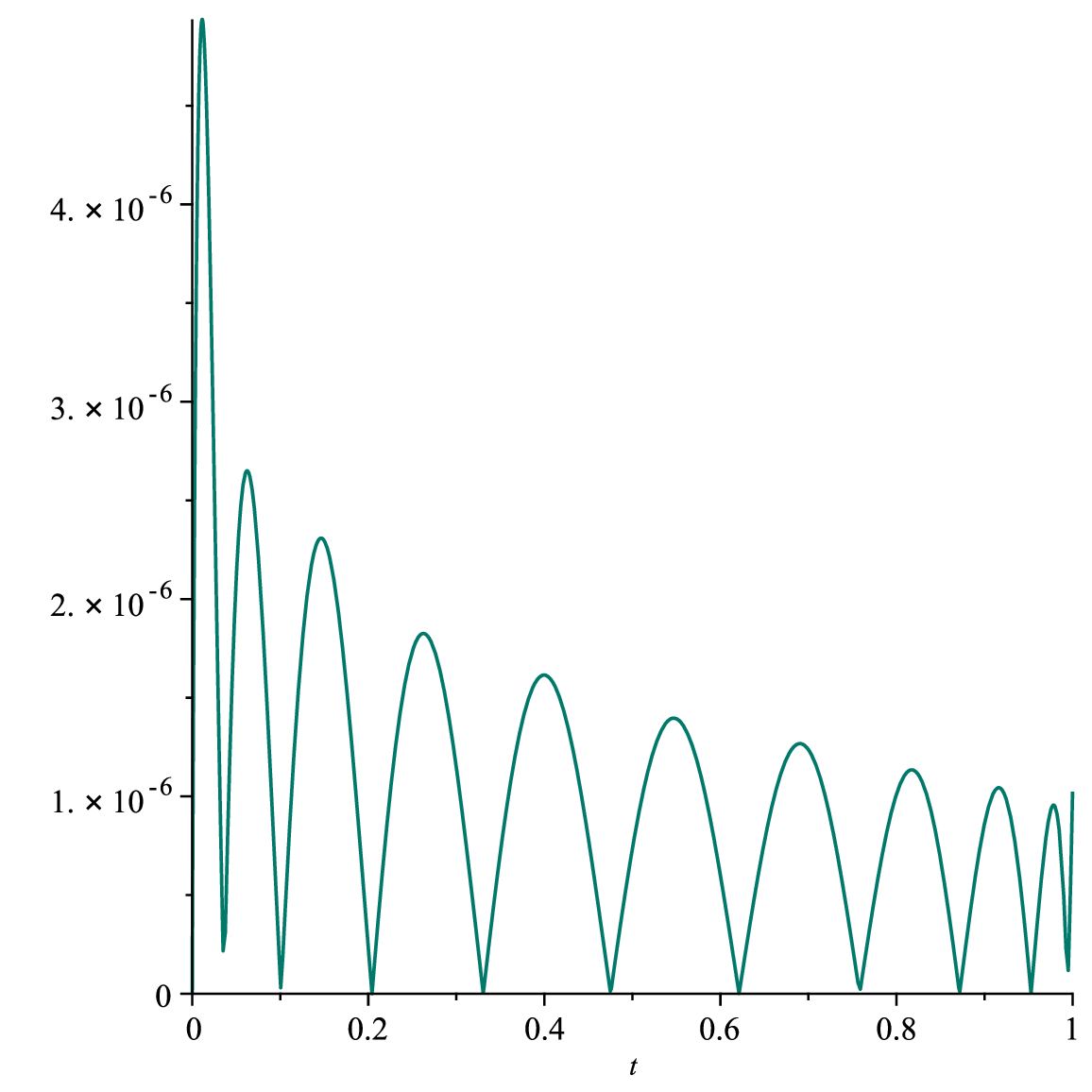}
			\caption{}

		\end{subfigure}
		\begin{subfigure}{.32\textwidth}
			\centering
			\includegraphics[width=1\linewidth]{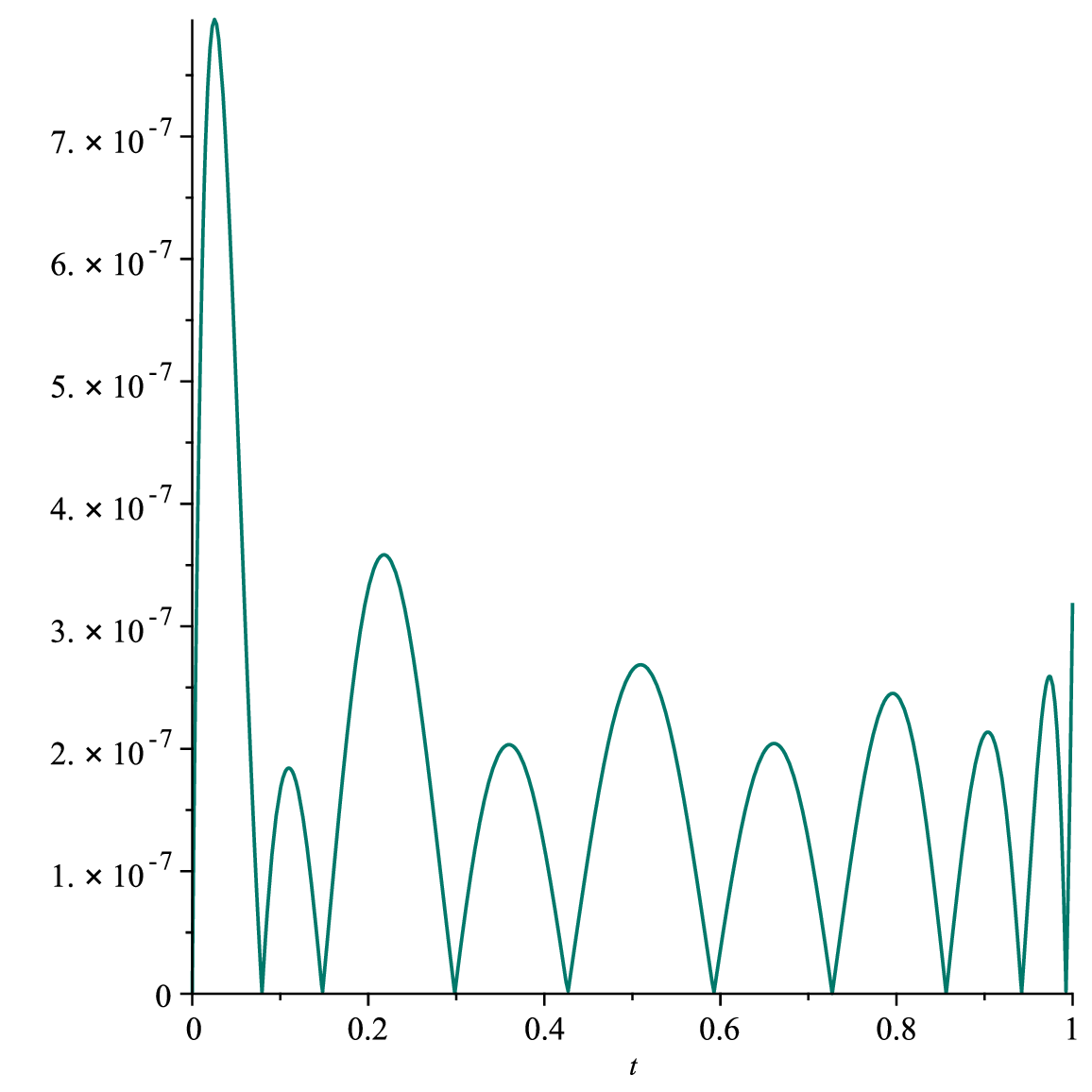}
			\caption{}

		\end{subfigure}
   \begin{subfigure}
        {.32\textwidth}
			\centering
			\includegraphics[width=1\linewidth]{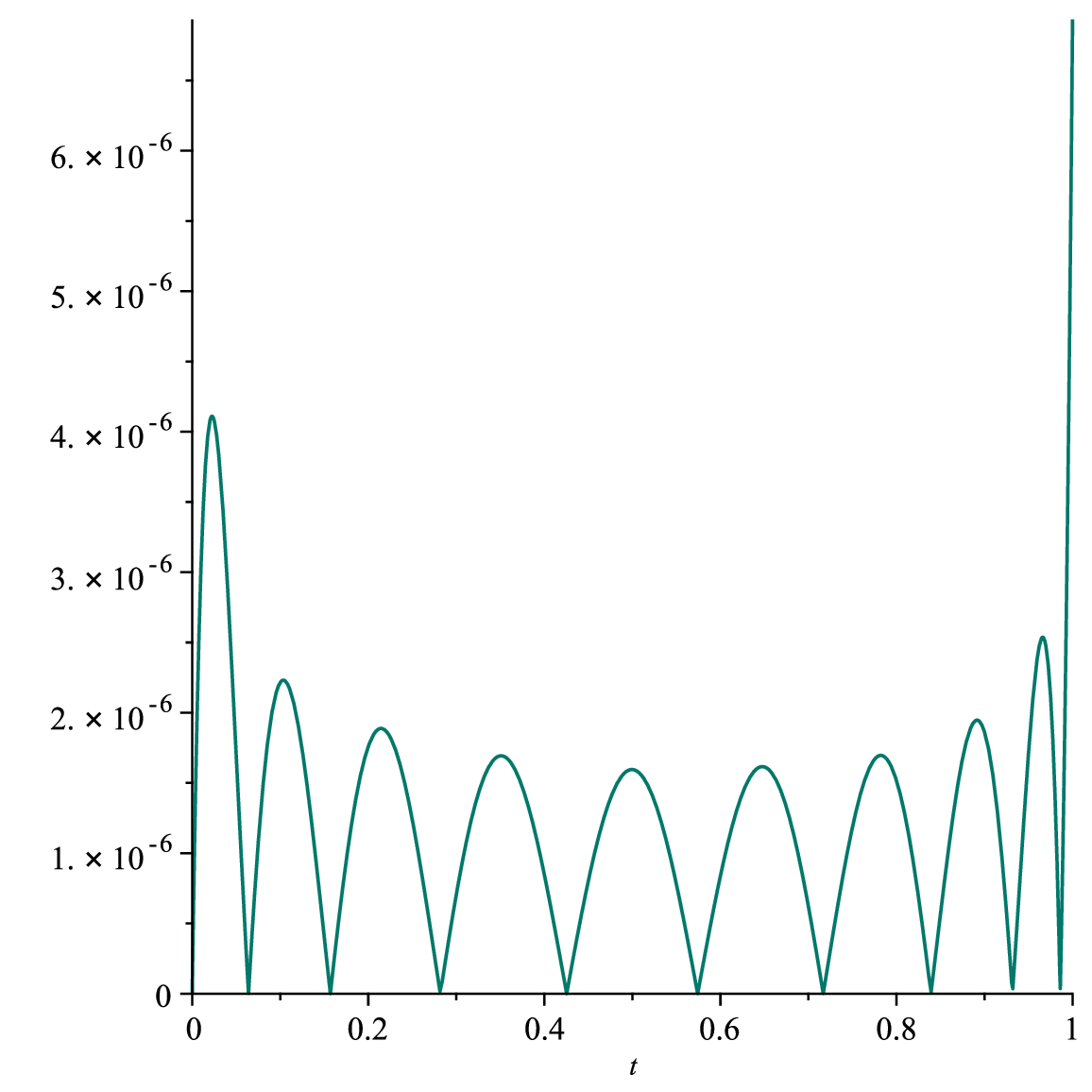}
			\caption{}

		\end{subfigure}
		\caption{(a) $|u_1(t)-\tilde{u}_1(t)|$, (b) $|u_2(t)-\tilde{u}_2(t)|$, and (c) $|u_3(t)-\tilde{u}_3(t)|$ for example \ref{ex3} with $m=10$.}\label{fig3}
        \end{figure}
\begin{table*}[]
\centering
 \captionsetup{font=scriptsize}
\caption{Exact, approximate solution, and relative error for $m=10$ on $\left[0,1\right]$ for example \ref{ex3}}.
\small
\resizebox{0.8\textwidth}{!}{
 \begin{tabular}
{|p{0.395\linewidth}|p{0.285\linewidth}|p{0.30\linewidth}|}

\begin{tabular}{@{}llll@{}}
\toprule
$t$ & $u_1(t)$ & $\Tilde{u}_1(t)$ & $\mathbb{E}_{u_1(t)}$\\ \midrule

{0.2} & 0.0178885 & 0.0178883 & $1.3 \times10^{-5}$\\
{0.4} & 0.1011928 & 0.10119127& $1.5 \times10^{-5}$  \\
{0.6} & 0.2788548 & 0.27885540 & $2.1 \times10^{-6}$  \\
{0.8} & 0.57243340 & 0.5724344 & $1.7 \times10^{-6}$ \\
{1} & 1 & 0.9999 & $1.0 \times10^{-6}$ \\
\bottomrule
\end{tabular}

&

\begin{tabular}{@{}lll@{}}
\toprule
 $u_2(t)$ & $\Tilde{u}_2(t)$ & $\mathbb{E}_{u_2(t)}$\\ \midrule
 0.04 & 0.0400003  & $8.2 \times10^{-6}$\\
 0.16 & 0.159999 & $7.4 \times10^{-7}$ \\
0.36 & 0.359999 & $9.9 \times10^{-8}$   \\
 0.64 & 0.6400002 & $3.8 \times10^{-7}$ \\
1 & 0.99999968 & $3.1 \times10^{-7}$ \\
\bottomrule
\end{tabular}
&
\begin{tabular}{@{}lll@{}}
\toprule
 $u_3(t)$ & $\Tilde{u}_3(t)$ & $\mathbb{E}_{u_3(t)}$\\ \midrule
  0.198669 & 0.198667  & $8.8 \times10^{-6}$ \\
 0.389418 & 0.389419  & $2.1 \times10^{-6}$ \\
 0.564642 & 0.564643 & $1.4 \times10^{-6}$ \\
0.717356 & 0.717354 & $2.1 \times10^{-6}$ \\
0.841470 & 0.841477  & $8.4 \times10^{-6}$ \\
\bottomrule
\end{tabular} \\
\hline
\end{tabular}}
\end{table*}\label{tab3}
        Figure \ref{fig3} illustrates the difference between the estimated solution and the exact solution for each $u_i(t)$. The relative error is computed across various intervals within the range of $[0,1]$, as indicated in Table \ref{tab3}. We compare our method with the waveform relaxation scheme proposed in \cite{ding2014waveform}. In Table \ref{tab3_1}, we present a comparison based on the $\ell_2$ norm error and the highest iteration of $N=20$, using the method described in \cite{ding2014waveform}.

      \begin{table}[h]
\centering
 \captionsetup{font=scriptsize}
\caption{Comparison of the obtained errors between our method and the method \cite{ding2014waveform}
for Example \ref{ex3} }
\label{tab3_1}
\begin{tabular}{lcc}
\hline
 & LS-SVR method &  method \cite{ding2014waveform} \\ \hline
$\|u_1(t)-\tilde{u}_1(t)\|_2$ & $2.2 \times10^{-6}$ & $\approx10^{-3}$ \\
$\|u_2(t)-\tilde{u}_2(t)\|_2$ & $5.3 \times10^{-7}$ & $\approx10^{-3}$ \\
$\|u_3(t)-\tilde{u}_3(t)\|_2$ & $7.4 \times10^{-6}$ & $\approx10^{-3}$ \\ \hline
\end{tabular}
\end{table}

\begin{example} \label{ex4}
Consider the following nonlinear fractional differential algebraic
equations
\begin{equation}
\begin{cases}

D^{\frac{1}{2}}u_1(t)+u_1(t)u_2(t)-u_3(t)=\frac{\Gamma(4)}{\Gamma(\frac{7}{2})}t^{\frac{5}{2}}+2t^4+t^7-e^t-tsin(t),
\\

D^{\frac{1}{2}}u_2(t)-\frac{\Gamma(5)}{\Gamma(\frac{9}{2})}t^{\frac{1}{2}}u_1(t)+2u_2(t)+u_1(t)u_3(t)={\frac{2}{\Gamma(\frac{3}{2})}}t^{\frac{1}{2}}+4t+2t^4+t^3e^t+t^4sin(t),
\\

u^2_1(t)-t^2u_2(t)+u_3(t)=e^t+2t^3+tsin(t), \quad t\in [0,1],
\\
\end{cases}
\end{equation}
	with the consistent initial conditions $u_1(0)=0, u_2(0)=0,$ and $u_3(0)=1$. The equation has analytical solutions $u_1(t)=t^3, u_2(t)=2t+t^4, u_3(t)=e^t+tsin(t)$.
  Table \ref{tab4} calculates the numerical results and the equation's absolute error value. Comparison of our reported results with those obtained in Table \ref{tab4_1} confirms the superiority and reliability of our proposed scheme over the method presented in \cite{ding2014waveform,ghanbari2018generalized} with the highest iteration of $N=16$.
 \end{example}
 \begin{table}[h]
\centering
 \captionsetup{font=scriptsize}
\caption{Comparison of the obtained errors between our method and methods in \cite{ding2014waveform,ghanbari2018generalized}
for Example \ref{ex4} }
\label{tab4_1}
\begin{tabular}{lccc}
\hline
 & LS-SVR method &  method \cite{ding2014waveform} &  method \cite{ghanbari2018generalized} \\ \hline
$\|u_1(t)-\tilde{u}_1(t)\|_2$ & $2.4  \times10^{-14}$ & $\approx10^{-3}$ & $2.3  \times10^{-7}$ \\
$\|u_2(t)-\tilde{u}_2(t)\|_2$ & $6.3 \times10^{-15}$ & $\approx10^{-3}$ &$9.0  \times10^{-5}$ \\
$\|u_3(t)-\tilde{u}_3(t)\|_2$ & $9.8 \times10^{-14}$ & $\approx10^{-4}$ &$4.6  \times10^{-5}$ \\ \hline
\end{tabular}
\end{table}

		\begin{figure}
		\centering
		\begin{subfigure}{.32\textwidth}
			\centering			\includegraphics[width=1\linewidth]{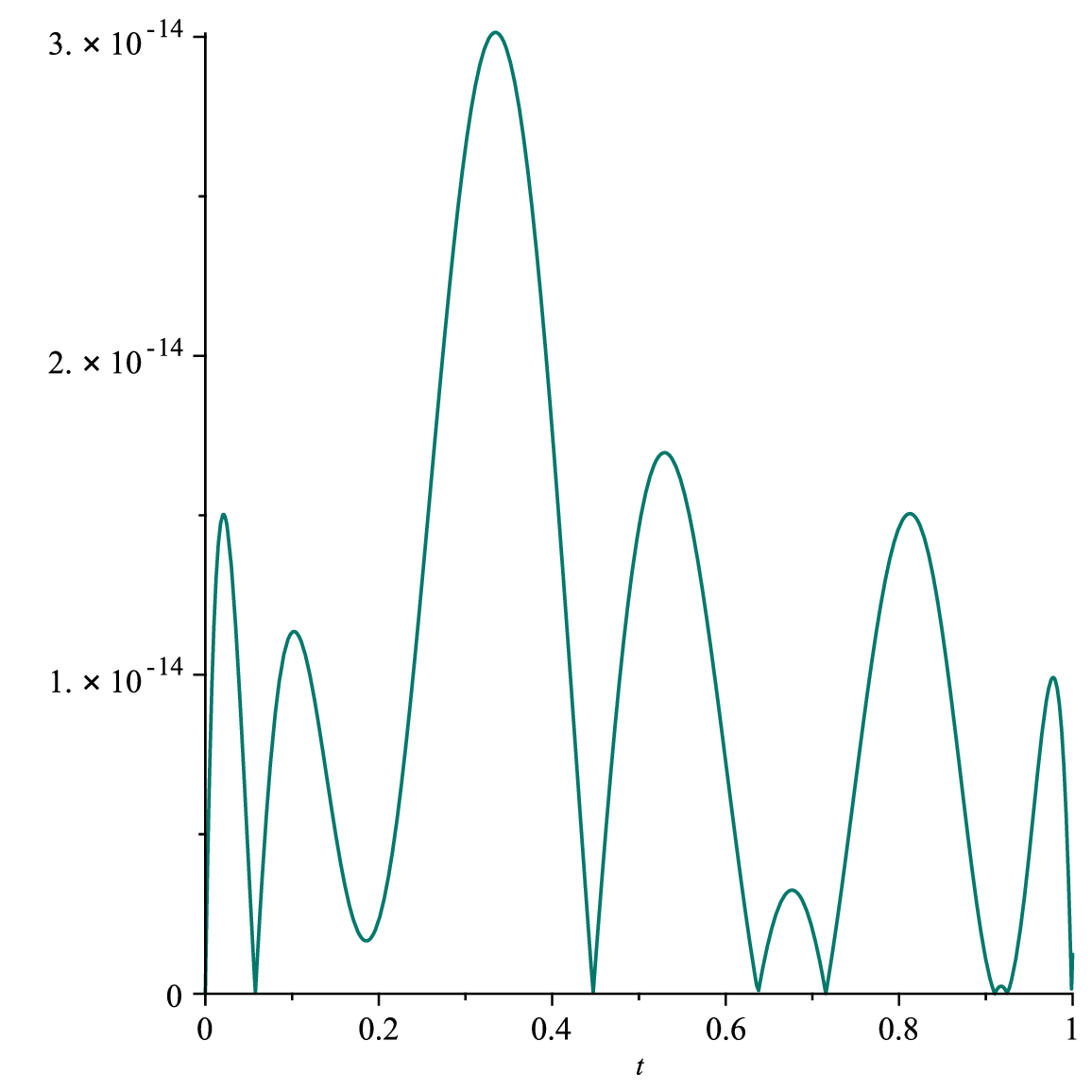}
			\caption{}

		\end{subfigure}
		\begin{subfigure}{.32\textwidth}
			\centering
			\includegraphics[width=1\linewidth]{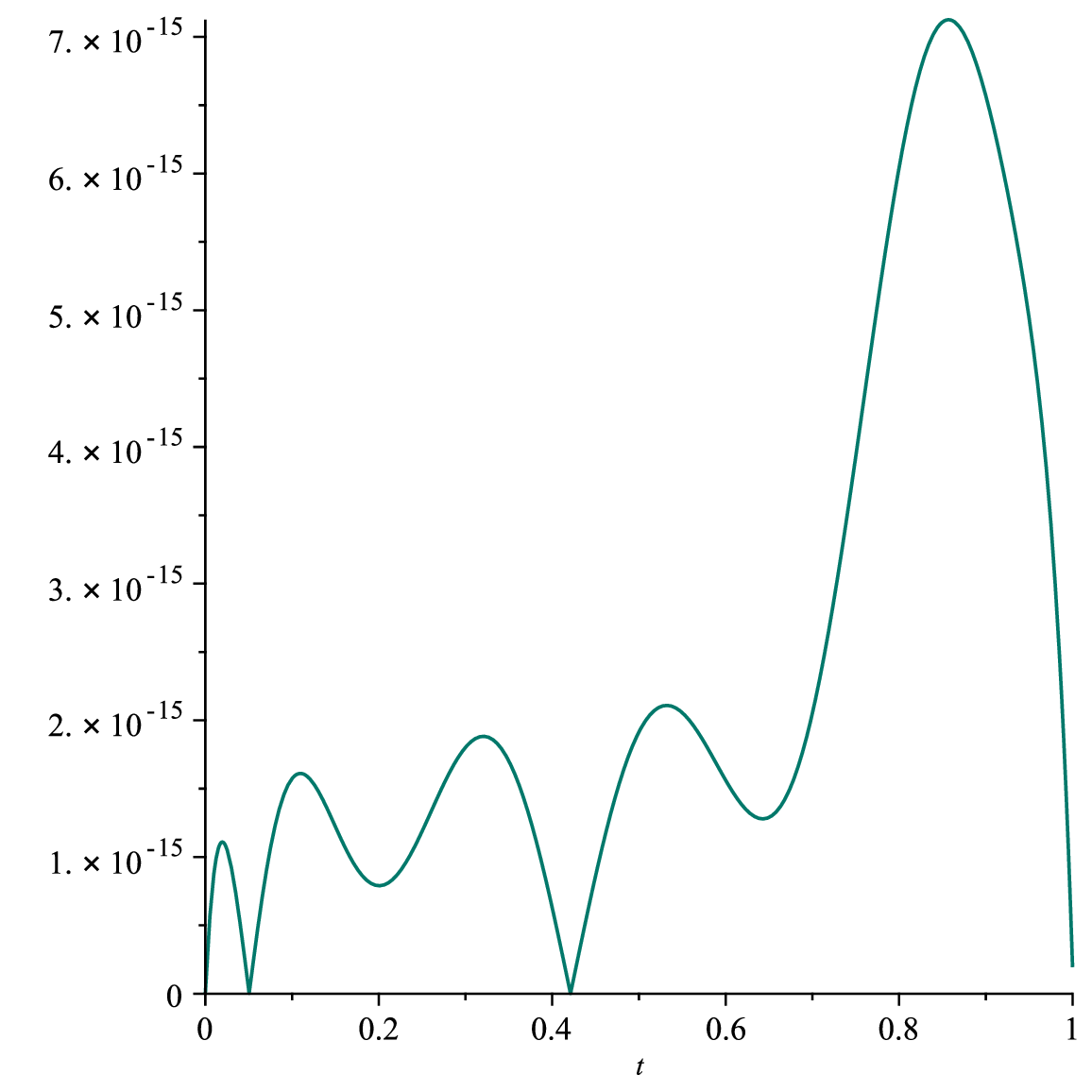}
			\caption{}

		\end{subfigure}
  		\begin{subfigure}{.32\textwidth}
			\centering
			\includegraphics[width=1\linewidth]{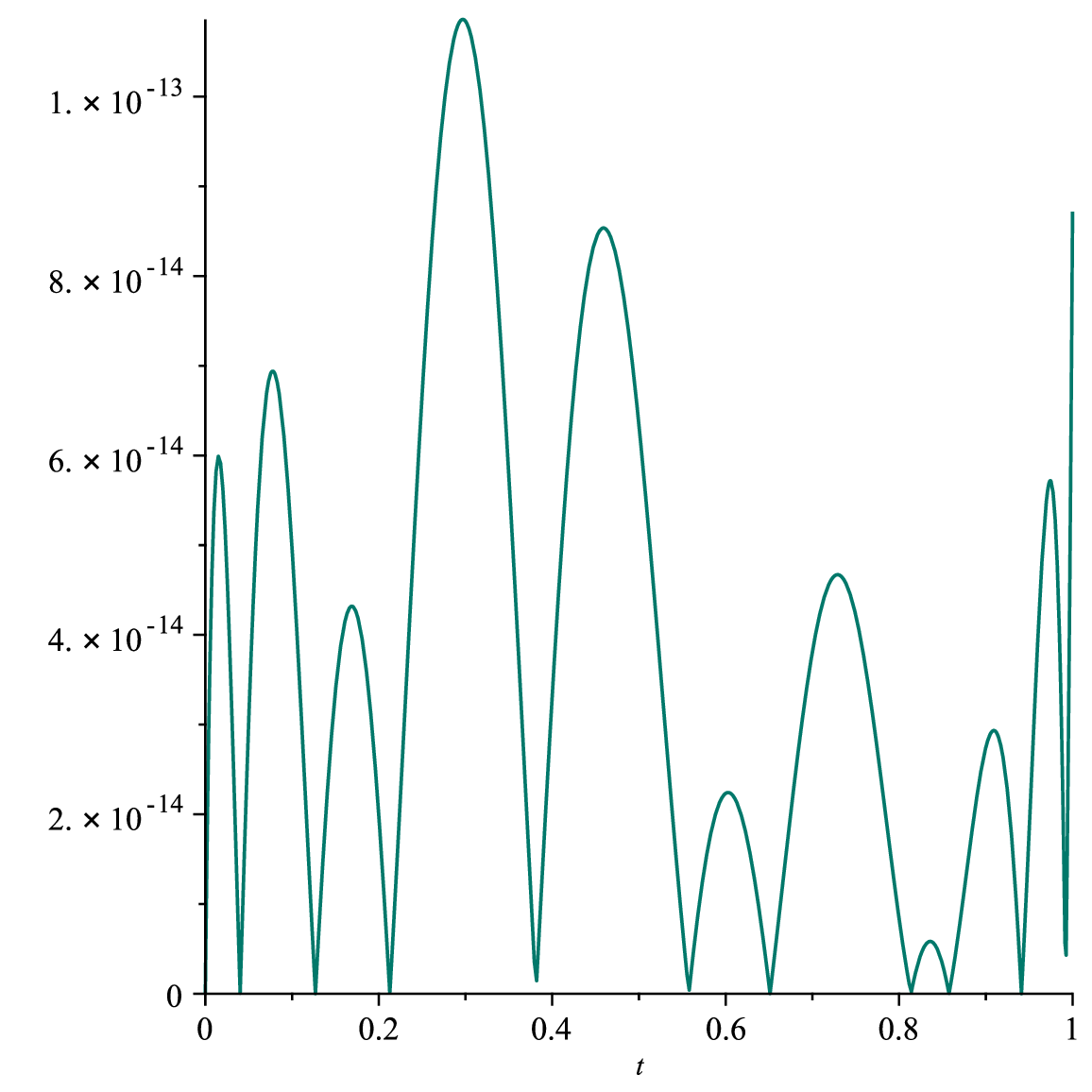}
			\caption{}

		\end{subfigure}\caption{(a) $|u_1(t)-\tilde{u}_1(t)|$, (b) $|u_2(t)-\tilde{u}_2(t)|$, and (c) $|u_3(t)-\tilde{u}_3(t)|$ for example \ref{ex4} with $m=10$.}\label{fig4}
	\end{figure}

 \begin{table*}[]
\centering
 \captionsetup{font=scriptsize}
\caption{Exact, approximate solution, and relative error for $m=10$ on $\left[0,1\right]$ for example \ref{ex4}}.
\small
\resizebox{0.9\textwidth}{!}{
 \begin{tabular}
{|p{0.33\linewidth}|p{0.30\linewidth}|p{0.48\linewidth}|}

\begin{tabular}{@{}llll@{}}
\toprule
$t$ & $u_1(t)$ & $\Tilde{u}_1(t)$ & $\mathbb{E}_{u_1(t)}$\\ \midrule

{0.2} & 0.008 & 0.007999 & $2.8 \times10^{-13}$ \\
{0.4} & 0.064 & 0.063999 &  $2.7 \times10^{-13}$   \\
{0.6} & 0.216 & 0.216000 &  $3.3 \times10^{-14}$   \\
{0.8} & 0.512 &0.512000 &  $2.8 \times10^{-14}$  \\
{1} & 1 & 0.999999 &  $1.2 \times10^{-15}$  \\
\bottomrule
\end{tabular}

&

\begin{tabular}{@{}lll@{}}
\toprule
 $u_2(t)$ & $\Tilde{u}_2(t)$ & $\mathbb{E}_{u_2(t)}$\\ \midrule
 0.4016  &0.4016000 &  $1.9 \times10^{-15}$  \\
 0.8256 & 0.8256000 &  $7.6 \times10^{-16}$   \\
 1.3296 & 1.3295999 &  $1.1 \times10^{-15}$    \\
 2.0096 & 2.0095999 &  $3.0 \times10^{-15}$  \\
 3 & 2.9999999       &  $6.4 \times10^{-17}$  \\
\bottomrule
\end{tabular}
&
\begin{tabular}{@{}lll@{}}
\toprule
 $u_3(t)$ & $\Tilde{u}_3(t)$ & $\mathbb{E}_{u_3(t)}$\\ \midrule
  1.2611366243191 & 1.2611366243192  &  $1.5 \times10^{-14}$  \\
1.64759203456473 & 1.64759203456476  &  $2.0 \times10^{-14}$   \\
 2.16090428442753 & 2.16090428442750 &  $1.0 \times10^{-14}$    \\
 2.79942580121208 & 2.79942580121209 &  $2.9 \times10^{-15}$ \\
3.559752813266 & 3.559752813267 &  $2.4 \times10^{-14}$  \\
\bottomrule
\end{tabular} \\
\hline
\end{tabular}}
\end{table*}\label{tab4}

\begin{example}\label{ex5}
Consider the partial differential-algebraic equation
\begin{equation}
\begin{cases}

D_tu_2(x,t)+D_tu_3(x,t)=-\frac{1}{2}x^2e^{-\frac{1}{2}t}+x^2cos(t),
\\

2D_tu_1(x,t)-D_tu_2(x,t)-D_tu_3(x,t)-u_2(x,t)=-2x^2e^{-t}-\frac{x^2e^{\frac{-t}{2}}}{2}-x^2cos(t),
\\

-D_{xx}u_3(x,t)+u_3(x,t)=-2sin(t)+x^2sin(t), \quad x\in[-0.5,0.5],\quad t\in[0,1],
\\
\end{cases}
\end{equation}
with the initial conditions:
\begin{center}
    $u_1(x,0)=x^2,D_tu_1(x,0)=-x^2,u_2(x,0)=x^2,D_tu_2(x,0)=\frac{-x^2}{2},u_3(x,0)=0,D_tu_3(x,0)=x^2,$
\end{center}
and whose analytical solution is $u(x,t)=\begin{bmatrix}x^2e^{-t} \\x^2e^{\frac{-t}{2}} \\x^2sin(t)\end{bmatrix}$. In Table \ref{tab5} the relative error is calculated for a number of points in the interval $[0,1]$. the absolute error  with $m=6$ and $\gamma=100$ is in Figure \ref{fig5}. Also, the relative error for various $m$ has been calculated in Table \ref{ta6}.
\end{example}

\begin{figure}[h]
		\centering
		\begin{subfigure}{.32\textwidth}
			\centering			\includegraphics[width=1\linewidth]{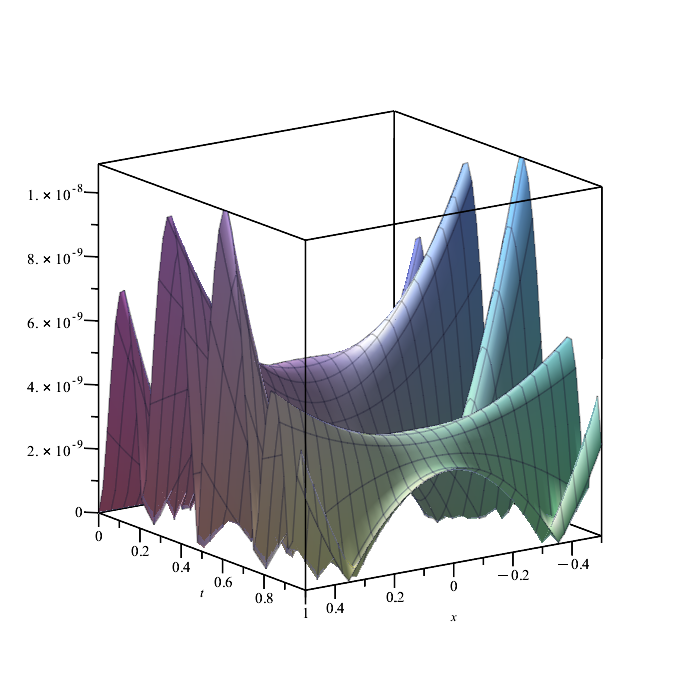}
			\caption{}

		\end{subfigure}
		\begin{subfigure}{.32\textwidth}
			\centering
			\includegraphics[width=1\linewidth]{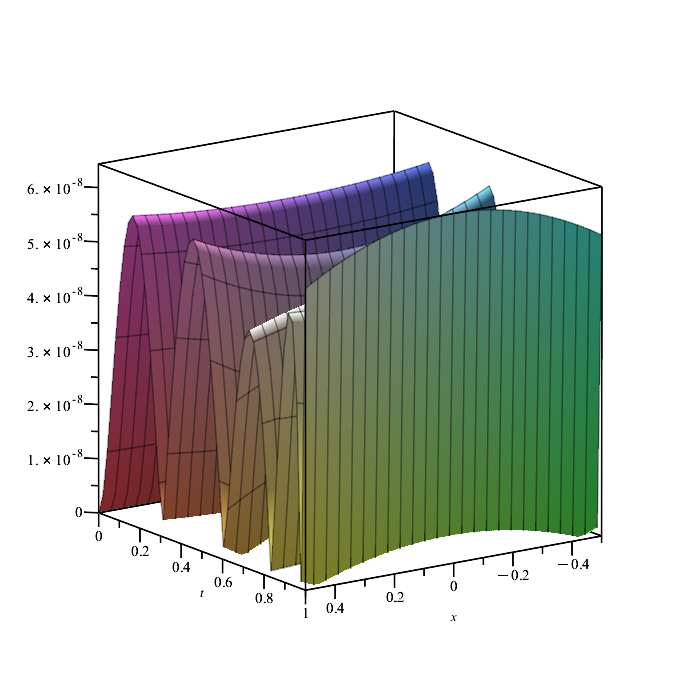}
			\caption{}

		\end{subfigure}
        \begin{subfigure}
        {.32\textwidth}
			\centering
			\includegraphics[width=1\linewidth]{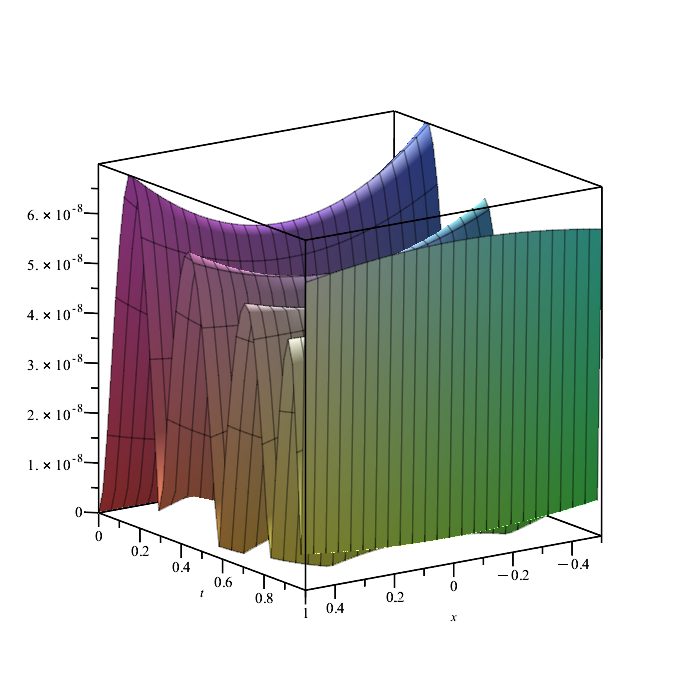}
			\caption{}

		\end{subfigure}
		\caption{(a) $|u_1(x,t)-\tilde{u}_1(x,t)|$, (b) $|u_2(x,t)-\tilde{u}_2(x,t)|$, and (c) $|u_3(x,t)-\tilde{u}_3(x,t)|$ for example \ref{ex5} with m=6.}\label{fig5}
	\end{figure}
 \begin{table*}[]
 \centering
  \captionsetup{font=scriptsize}

\caption{Exact, approximate solution, and relative error of example \ref{ex5} with $m=6$. }

\resizebox{0.9\textwidth}{!}{
 \begin{tabular}
{|p{0.56\linewidth}|p{0.38\linewidth}|p{0.41\linewidth}|}

\begin{tabular}{@{}lllll@{}}
\toprule
$t$ & $x$ &
$u_1(x,t)$ & $\Tilde{u}_1(x,t)$ & $\mathbb{E}_{u_1(x,t)}$\\ \midrule

{0.02} & 0.02 &
0.0003920794  & 0.0003920795 &$3.2 \times10^{-7}$ \\
{0.04} & 0.04 &
0.0015372631 &0.0015372635 & $2.9 \times10^{-7}$  \\
{0.06} & 0.06 &
0.00339035232 & 0.0033903531 & $2.5 \times10^{-7}$   \\
{0.08} & 0.08 &
0.005907944 & 0.005907945 & $2.1 \times10^{-7}$ \\
{0.1} & 0.1 &
0.009048374 & 0.009048375& $1.8 \times10^{-7}$ \\
\bottomrule
\end{tabular}

&

\begin{tabular}{@{}lll@{}}
\toprule
 $u_2(x,t)$ & $\Tilde{u}_2(x,t)$ & $\mathbb{E}_{u_2(x,t)}$\\ \midrule
 0.00039601 & 0.00039602  & $7.8 \times10^{-6}$ \\
0.00156831 &0.00156832& $6.8 \times10^{-6}$\\
0.00349360 & 0.00349362 & $5.8 \times10^{-6}$   \\
0.00614905 & 0.00614908  & $4.9 \times10^{-6}$ \\
0.00951229 & 0.0095123 & $4.2 \times10^{-6}$ \\
\bottomrule
\end{tabular}
&
\begin{tabular}{@{}lll@{}}
\toprule
 $u_3(x,t)$ & $\Tilde{u}_3(x,t)$ & $\mathbb{E}_{u_3(x,t)}$\\ \midrule
 0.0000079994  & 0.000007996 & $3.9\times10^{-4}$ \\
 0.00006398 & 0.00006397 & $1.6 \times10^{-4}$  \\
 0.00021587 & 0.00021584 & $9.5 \times10^{-5}$   \\
 0.00051145 & 0.00051142 & $6.0 \times10^{-5}$ \\
0.0009983& 0.0009982 & $4.0 \times10^{-5}$\\
\bottomrule
\end{tabular} \\
\hline
\end{tabular}}
\end{table*}\label{tab5}
\begin{table}[]\label{tab6}
    \centering
    \captionsetup{font=scriptsize}
    \caption{Relative error of example \ref{ex5} with different $m$ for $(x,t)=(.02,0.02),(.04,.04),(.06,.06),(.08,.08), and (.1,.1)$}
    \label{ta6}

\begin{tabular}{|c|c|c|c|}

  \hline
  $m$ & $\mathbb{E}_{u_1(x,t)}$ & $\mathbb{E}_{u_2(x,t)}$ & $\mathbb{E}_{u_3(x,t)}$ \\
  \hline
  $m=8$ & \multirow{5}{*}{\begin{tabular}{@{}c@{}}$2.9\times10^{-9}$ \\ $2.1\times10^{-9}$ \\ $1.4\times10^{-9}$ \\ $8.7\times10^{-10}$\\ $4.0\times10^{-10}$\end{tabular}} & \multirow{5}{*}{\begin{tabular}{@{}c@{}}$3.9\times10^{-8}$\\ $1.9\times10^{-8}$ \\ $5.1\times10^{-9}$ \\ $5.2\times10^{-9}$ \\ $1.2\times10^{-8}$\end{tabular}} & \multirow{5}{*}{\begin{tabular}{@{}c@{}}$1.9\times10^{-6}$ \\ $4.8\times10^{-7}$ \\ $8.3\times10^{-8}$ \\ $6.3\times10^{-8}$ \\ $1.1\times10^{-7}$\end{tabular}} \\

  & & & \\

   & & & \\

   & & & \\

   & & & \\

  \hline
  $m=10$ & \multirow{5}{*}{\begin{tabular}{@{}c@{}}$5.2\times10^{-13}$ \\ $3.6\times10^{-13}$ \\ $2.5\times10^{-13}$ \\ $1.7\times10^{-13}$ \\ $1.1\times10^{-13}$\end{tabular}} & \multirow{5}{*}{\begin{tabular}{@{}c@{}}$2.8\times10^{-11}$ \\ $1.7\times10^{-11}$ \\ $9.5\times10^{-12}$ \\ $4.6\times10^{-12}$\\ $1.5\times10^{-12}$\end{tabular}} & \multirow{5}{*}{\begin{tabular}{@{}c@{}}$1.3\times10^{-9}$\\$4.2\times10^{-10}$ \\ $1.5\times10^{-10}$ \\ $5.5\times10^{-11}$ \\ $1.5\times10^{-11}$\end{tabular}} \\

   & & & \\

   & & & \\

   & & & \\

   & & & \\
  \hline
  $m=12$ & \multirow{5}{*}{\begin{tabular}{@{}c@{}}$1.7\times10^{-15}$ \\ $1.0\times10^{-15}$ \\ $5.7\times10^{-16}$ \\ $3.0\times10^{-16}$ \\ $1.5\times10^{-16}$\end{tabular}} & \multirow{5}{*}{\begin{tabular}{@{}c@{}}$1.2\times10^{-13}$ \\ $5.5\times10^{-14}$\\ $1.8\times10^{-14}$ \\ $1.4\times10^{-15}$ \\ $5.0\times10^{-15}$\end{tabular}} & \multirow{5}{*}{\begin{tabular}{@{}c@{}}$6.1\times10^{-12}$ \\$1.3\times10^{-12}$ \\ $3.0\times10^{-13}$ \\ $1.6\times10^{-14}$ \\ $4.8\times10^{-14}$\end{tabular}} \\

   & & & \\

  & & & \\

  & & & \\

   & & & \\
  \hline

\end{tabular}
\end{table}

 \section{Conclusion} \label{sec5}
 The primary aim of this research is to establish a cohesive relationship among a machine learning model, weighted residual methods, and Legendre's orthogonal polynomials to effectively address DAEs. We adopted a system of differential-algebraic equations and employed LS-SVR to approximate their solutions. We represented the solution as a linear combination of orthogonal Legendre polynomials with unknown coefficients and subsequently formulated the residual function. This residual function was aligned with the principles of weighted residual methods by establishing its inner product with the Dirac delta function. The resulting system of equations was then adopted as the constraints for the quadratic programming problem. This approach presents notable advantages in contrast to conventional methods, including the utilization of regularization techniques for ill-conditioned problems, attainment of a singular solution with exponential convergence rate, and formulation of a quadratic programming problem leading to a symmetric positive definite matrix within a dual space. These attributes collectively contribute to the inherent simplicity of this technique, rendering it suitable for solving diverse categories of DAEs. Through the examples presented, we have applied LS-SVR for solving the Volterra integral DAE of fractional order, partial DAE, and fractional DAEs with high accuracy for solving equations. Nonetheless, this approach encounters challenges in solving DAEs with non-periodic boundary conditions. Addressing the issue within the current framework is complex and requires additional research to improve and broaden the method. Within this model, we achieved precise outcomes through the application of widely recognized Legendre orthogonal functions as the model's kernel. For future investigations, alternative kernel functions like Chebyshev, Hermite, and Gegenbauer polynomials hold potential for exploration. Furthermore, the scope of this method can be expanded to encompass the resolution of other differential equation categories.


\section{Declarations}
Declarations of interest: none

\section{Data availability statement}
not applicable.

\bibliography{references}

\end{document}